\documentclass[a4paper,10pt,reqno]{amsart}
\usepackage{mathtools}
\usepackage{multirow}
\usepackage{enumerate}
\usepackage{enumitem}
\usepackage{amsmath,amssymb,amsthm,mathrsfs,amsfonts,float}
\usepackage[all]{xy}
\usepackage{tikz}
\usepackage{graphicx}
\usepackage{caption}
\usepackage{hyperref}
\usepackage{bbm}
\usepackage{todonotes}
\usepackage{afterpage}
\usepackage{xcolor}
\hypersetup{
	colorlinks,
	linkcolor={black},
	citecolor={blue},
	urlcolor={black}
}
\usepackage{longtable}

\numberwithin{equation}{section}	  
\theoremstyle{definition}
\newtheorem{thm}{Theorem}[section]
\newtheorem{prop}[thm]{Proposition}
\newtheorem*{prop-nonum}{Proposition}
\newtheorem{lem}[thm]{Lemma}
\newtheorem*{lem-nonum}{Lemma}

\newtheorem*{defn-nonum}{Definition}

\newtheorem{rem}[thm]{Remark}
\newtheorem*{rem-nonum}{Remark}

\newtheorem*{conj-nonum}{Conjecture}

\newtheorem*{fact-nonum}{Fact}

\newenvironment{myitemize}
{	\begin{itemize}
		\setlength{\itemsep}{0pt}
		\setlength{\parskip}{0pt}
		\setlength{\parsep}{0pt}	}
	{	\end{itemize}	}

\makeatletter
\providecommand*{\twoheadrightarrowfill@}{%
	\arrowfill@\relbar\relbar\twoheadrightarrow
}
\providecommand*{\twoheadleftarrowfill@}{%
	\arrowfill@\twoheadleftarrow\relbar\relbar
}
\providecommand*{\xtwoheadrightarrow}[2][]{%
	\ext@arrow 0579\twoheadrightarrowfill@{#1}{#2}%
}
\providecommand*{\xtwoheadleftarrow}[2][]{%
	\ext@arrow 5097\twoheadleftarrowfill@{#1}{#2}%
}
\makeatother

\makeatletter
\newbox\xrat@below
\newbox\xrat@above
\newcommand{\xrightarrowtail}[2][]{%
	\setbox\xrat@below=\hbox{\ensuremath{\scriptstyle #1}}%
	\setbox\xrat@above=\hbox{\ensuremath{\scriptstyle #2}}%
	\pgfmathsetlengthmacro{\xrat@len}{max(\wd\xrat@below,\wd\xrat@above)+.6em}%
	\mathrel{\tikz [>->,baseline=-.75ex]
		\draw (0,0) -- node[below=-2pt] {\box\xrat@below}
		node[above=-2pt] {\box\xrat@above}
		(\xrat@len,0) ;}}
\makeatother

\makeatletter
\newcommand*{\relrelbarsep}{.386ex}
\newcommand*{\relrelbar}{%
	\mathrel{%
		\mathpalette\@relrelbar\relrelbarsep
	}%
}
\newcommand*{\@relrelbar}[2]{%
	\raise#2\hbox to 0pt{$\m@th#1\relbar$\hss}%
	\lower#2\hbox{$\m@th#1\relbar$}%
}
\providecommand*{\rightrightarrowsfill@}{%
	\arrowfill@\relrelbar\relrelbar\rightrightarrows
}
\providecommand*{\leftleftarrowsfill@}{%
	\arrowfill@\leftleftarrows\relrelbar\relrelbar
}
\providecommand*{\leftrightarrowsfill@}{%
	\arrowfill@\leftrightarrows\relrelbarlr\relrelbarlr
}
\providecommand*{\xrightrightarrows}[2][]{%
	\ext@arrow 0359\rightrightarrowsfill@{#1}{#2}%
}
\providecommand*{\xleftleftarrows}[2][]{%
	\ext@arrow 3095\leftleftarrowsfill@{#1}{#2}%
}
\makeatother

\newcommand{\GL}{\mathrm{GL}}

\newcommand{\pp}{\partial}

\newcommand{\Z}{\mathbb{Z}}

\newcommand{\Q}{\mathbb{Q}}

\newcommand{\im}{\mathrm{im}}
\newcommand{\ra}{\rightarrow}

\makeatletter
\MHInternalSyntaxOn
\def\MT_leftarrow_fill:{%
	\arrowfill@\leftarrow\relbar\relbar}
\def\MT_rightarrow_fill:{%
	\arrowfill@\relbar\relbar\rightarrow}
\newcommand{\xrightleftarrows}[2][]{\mathrel{%
		\raise.55ex\hbox{%
			$\ext@arrow 0359\MT_rightarrow_fill:{\phantom{#1}}{#2}$}%
		\setbox0=\hbox{%
			$\ext@arrow 3095\MT_leftarrow_fill:{#1}{\phantom{#2}}$}%
		\kern-\wd0 \lower.55ex\box0}}
\MHInternalSyntaxOff
\makeatother

\makeatletter
\newlength\min@xx
\newcommand*\xxra[1]{\begingroup
  \settowidth\min@xx{$\m@th\scriptstyle#1$}
  \@xxra}
\newcommand*\@xxra[2][]{
  \sbox8{$\m@th\scriptstyle#1$}  
  \ifdim\wd8>\min@xx \min@xx=\wd8 \fi
  \sbox8{$\m@th\scriptstyle#2$} 
  \ifdim\wd8>\min@xx \min@xx=\wd8 \fi
  \xrightarrow[{\mathmakebox[\min@xx]{\scriptstyle#1}}]
    {\mathmakebox[\min@xx]{\scriptstyle#2}}
  \endgroup}
\makeatother

\makeatletter
\newcommand*\xxla[1]{\begingroup
  \settowidth\min@xx{$\m@th\scriptstyle#1$}
  \@xxla}
\newcommand*\@xxla[2][]{
  \sbox8{$\m@th\scriptstyle#1$}  
  \ifdim\wd8>\min@xx \min@xx=\wd8 \fi
  \sbox8{$\m@th\scriptstyle#2$} 
  \ifdim\wd8>\min@xx \min@xx=\wd8 \fi
  \xleftarrow[{\mathmakebox[\min@xx]{\scriptstyle#1}}]
    {\mathmakebox[\min@xx]{\scriptstyle#2}}
  \endgroup}
\makeatother

\usepackage[a4paper, total={6.2in, 9in}, left=1.025in, right=1.025in, bottom=1.025in]{geometry}

\usepackage[T1]{fontenc}
\usepackage[utf8]{inputenc}
\usepackage{palatino}
\usepackage{textcomp}

\graphicspath{{figures/}}
\usepackage{float}

\usepackage{cleveref}
\usepackage[style=numeric,backend=bibtex,maxbibnames=99,firstinits=true,sorting=nty]{biblatex}
\DeclareFieldFormat{postnote}{#1}
\allowdisplaybreaks

\title{On slice alternating 3-braid closures}
\author{Vitalijs Brejevs}
\address{School of Mathematics and Statistics, University of Glasgow, Glasgow, United Kingdom}
\email{Vitalijs.Brejevs@glasgow.ac.uk}

\begin{document}
	\begin{abstract}
		We construct ribbon surfaces of Euler characteristic one for several infinite families of alternating 3-braid closures. We also use a twisted Alexander polynomial obstruction to conclude the classification of smoothly slice knots which are closures of alternating 3-braids with up to 20 crossings.
	\end{abstract}
	\maketitle
	
	\section{Introduction}

	By an \emph{alternating braid} we mean a braid such that along any strand, over- and undercrossings alternate. Let $ \sigma_1 $ and $ \sigma_2 $ be the standard generators of the braid group on three strands $ B_3 $. If the closure of an alternating 3-braid has non-zero determinant, then it is isotopic to the closure of a braid
		\[
		\tag{$\star$}
		\sigma_1^{a_1} \sigma_2^{-b_1} \sigma_1^{a_2} \sigma_2^{-b_2} \dots \sigma_1^{a_n} \sigma_2^{-b_n}
		\]
	with $ n \geqslant 1 $ for some $ a_i $, $ b_i \geqslant 1 $ for all $ i $. Every 3-braid of the form $ (\star) $ can be equivalently described by its \emph{associated string}
		$
			\mathbf{a} = (2^{[a_1 - 1]}, b_1 + 2, \dots, 2^{[a_n - 1]}, b_n + 2),
		$
	where $ 2^{[a_i - 1]} $ represents the substring consisting of the number~$ 2 $ repeated $ a_i - 1 $ times. Cyclic rotations and reversals of $ \mathbf{a} $ do not change the isotopy class of respective braid closures in $ S^3 $, so we consider associated strings up to those two operations. The \emph{linear dual} of a string $ \mathbf{b} = (b_1, \dots, b_k) $ with all $ b_i \geqslant 2 $ is defined as follows: if $ b_j \geqslant 3 $ for some $ j $, write $ \mathbf{b} $ in the form
		$
			\mathbf{b} = (2^{[m_1]}, 3 + n_1, 2^{[m_2]}, 3 + n_2, \dots, 2^{[m_j]}, 2 + n_j)
		$
	with $ m_i, n_i \geqslant 0 $ for all $ i $. Then its linear dual is
		$
			\mathbf{c} = (2 + m_1, 2^{[n_1]}, 3 + m_2, 2^{[n_2]}, 3 + m_3, \dots, 3 + m_j, 2^{[n_j]})
		$.
	If $ \mathbf{b} $ is $ ( 2^{[k]} ) $ or $ (1) $, define its linear dual as $ (k+1) $ or the empty string, respectively.

	Given a link $ L \subset S^3 $, by a \emph{ribbon surface} we mean a surface $ F $ bounded by~$ L $ that is properly smoothly embedded in $ D^4 $, has no closed components, and may be isotoped rel boundary so that the radial distance function $ D^4 \ra [0, 1] $ induces a handle decomposition on $ F $ with only 0- and 1-handles. By a \emph{slice surface} we mean a surface $ S $ bounded by $ L $ that is properly smoothly embedded in $ D^4 $ and has no closed components; neither $ F $ nor $ S $ are required to be connected or orientable. Following~\cite{DonaldOwens2012}, we say that $ L $ which bounds a ribbon (or slice) surface of Euler characteristic one is \emph{$\chi$-ribbon} (or \emph{$\chi$-slice}); these definitions coincide with the usual definitions of ribbon and slice in the case of knots. Clearly, if $ L $ is $ \chi $-ribbon, then it is also $ \chi $-slice.

	In~\cite{simone2020classification}, Simone has classified associated strings of all alternating 3-braid closures $ L $ with non-zero determinant such that $ \Sigma_2(S^3, L) $, the double branched cover of $ S^3 $ over $ L $, is unobstructed by Donaldson's theorem from bounding a rational ball, into five families:
		\begin{itemize}
			\item $ \mathcal{S}_{2a}=\{(b_1+3,b_2,\ldots,b_k,2,c_l,\ldots,c_1)\}$;
			\item $ \mathcal{S}_{2b}=\{(3+x,b_1,\ldots,b_{k-1},b_k+1,2^{[x]},c_l+1,c_{l-1},\ldots,c_1)\,|\, x\geqslant 0\text{ and } k+l\geqslant2\}$;
			\item $\begin{aligned}[t] \mathcal{S}_{2c}=\{&(3+x_1,2^{[x_2]},3+x_3,2^{[x_4]},\ldots,3+x_{2k+1},2^{[x_1]}, 3+x_2,2^{[x_3]},\ldots,3+x_{2k},2^{[x_{2k+1}]})\,|\\ &\, k\geqslant0 \text{ and } x_i\geqslant 0 \text{ for all } i\};\end{aligned}$
			\item $ \mathcal{S}_{2d}=\{(2,2+x,2,3,2^{[x-1]},3,4)\,|\, x \geqslant 1\} \cup \{ (2, 2, 2, 4, 4) \}$;
			\item $ \mathcal{S}_{2e}=\{(2,b_1+1,b_2,\ldots,b_k,2,c_l,\ldots,c_2,c_1+1,2)\,|\, k+l \geqslant 3 \} \cup \{ (2,2,2,3) \}$.
		\end{itemize}
	Here strings $ (b_1, \dots, b_k) $ and $ (c_1, \dots, c_l) $ are linear duals of each other. Since $ \Sigma_2(S^3, L) $ of a $\chi$-slice link~$ L $ bounds a rational ball~\cite[Proposition 2.6]{DonaldOwens2012}, every $ \chi $-slice alternating 3-braid closure with non-zero determinant has its associated string in one of these families. Moreover, Simone has explicitly constructed rational balls for all such alternating 3-braid closures.

	We show that alternating 3-braid closures whose associated strings lie in $ \mathcal{S}_{2a} \cup \mathcal{S}_{2b} \cup \mathcal{S}_{2d} \cup \mathcal{S}_{2e} $ are $ \chi $-ribbon by exhibiting band moves, defined in Section~2, which make their link diagrams isotopic to the two- or three-component unlink. In Section~3, we consider the set $ \mathcal{S}_{2c} \setminus (\mathcal{S}_{2a} \cup \mathcal{S}_{2b} \cup \mathcal{S}_{2d} \cup \mathcal{S}_{2e}) $ that includes strings associated to known non-$ \chi $-slice alternating 3-braid closures, such as certain Turk's head knots, and list more examples of potentially non-$ \chi $-slice knots and links. In Section~4 we follow~\cite{heraldkirklivingston} and~\cite{ammmps2020branched} in applying a twisted Alexander polynomial obstruction to show that among these examples, three knots are indeed not slice; this concludes the classification of smoothly slice knots which are closures of alternating 3-braids with up to 20 crossings. \\

	\noindent\textbf{Acknowledgements.} We thank our supervisors Brendan Owens and Andy Wand for many helpful discussions and advice, Frank Swenton for his \textsf{KLO} software which we have found very useful, and Ross Paterson for consultation on some algebra. We are grateful to Paolo Aceto, Marco Golla, Kyle Larson, Ana Lecuona, Paolo Lisca, Allison Miller, Maggie Miller and Jonathan Simone for their comments on an earlier draft of this paper. This work was funded by the Carnegie Trust.

	\section{Ribbon surfaces for $\mathcal{S}_{2a} \cup \mathcal{S}_{2b} \cup \mathcal{S}_{2d} \cup \mathcal{S}_{2e}$}
	\label{sec:ribbons}

	One may exhibit a ribbon surface for a link $ L $ as follows. By a \emph{band move} on $ L $ we mean choosing an embedding $ \varphi: D^1 \times D^1 \hookrightarrow S^3 $ of a \emph{band} so that the image of $ \varphi $ is disjoint from $ L $ except for $ \varphi(\pp D^1 \times D^1) $ coincident with two segments of $ L $, removing those segments, joining corresponding ends along $ \varphi(D^1 \times \pp D^1) $ and smoothing the corners. This operation amounts to removing a 1-handle in the putative ribbon surface $ F $. If after $ n $ band moves, the resulting link is isotopic to the $ (n+1) $-component unlink, one has indeed obtained a ribbon surface $ F $ of Euler characteristic one bounded by $ L $, since each component of the unlink bounds a 0-handle of $ F $. Each band may be represented on a link diagram by an arc with endpoints on $ L $ that crosses the strands of $ L $ transversally, has no self-crossings, and is annotated by the number of half-twists in the band relative to the blackboard framing.

	Given a 3-braid $ \beta = \sigma_1^{a_1} \sigma_2^{-b_1} \dots \sigma_1^{a_n} \sigma_2^{-b_n} $, we draw it from left to right, as shown in Figure~\ref{fig:closure}, and orient all strings in the closure $ \widehat{\beta} $ clockwise. Choose the chessboard colouring of the diagram for $ \widehat{\beta} $ where the unbounded region is white. Then there are $ m = \left(\sum_{i=1}^n a_i \right) + 1 $ black regions. We can index the black regions, excluding the one not adjacent to the unbounded region (marked by $ * $ in Figure~\ref{fig:closure}), by $ \{ 1, \dots, m - 1 \} $ such that the number of crossings along the boundary of the region indexed by $ i $ is given by the $ i^{\textrm{th}} $ entry of the associated string $ \mathbf{a} = (2^{[a_1 - 1]}, b_1 + 2, \dots, 2^{[a_n - 1]}, b_n + 2) $, and the region indexed by $ i $ shares one crossing with each of the regions indexed by $ i - 1 $ and $ i + 1 $ (mod $ m - 1 $).

	\begin{figure}[h]
		\centering
		\def\svgwidth{0.6\textwidth}
		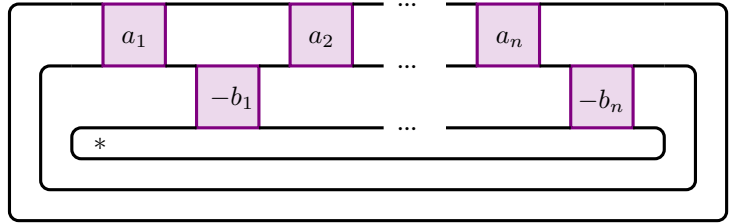
		\caption{
			A generic alternating 3-braid closure. We denote sequences of positive (negative) crossings by blocks annotated by positive (negative) coefficients.
		} \label{fig:closure}
	\end{figure}

	\begin{prop}
	\label{prop:ribbons}
		Let $ \mathbf{a} $ be the associated string of an alternating 3-braid closure $ \widehat{\beta} $. If $ \mathbf{a} \in \mathcal{S}_{2a} \cup \mathcal{S}_{2d} \cup \mathcal{S}_{2e} $, then $ \widehat{\beta} $ bounds a ribbon surface with a single 1-handle. If $ \mathbf{a} \in \mathcal{S}_{2b} $, then $ \widehat{\beta} $ bounds a ribbon surface with at most two 1-handles.
	\end{prop}

	Our main observation, previously used by Lisca~\cite{lisca-rational} and Lecuona~\cite{lecuona-montesinos}, is that if $ \mathbf{a} $ contains two disjoint linearly dual substrings (possibly perturbed on the ends), then the link diagram of $ \widehat{\beta} $ contains sub-braids which, if connected to each other by a half-twist $ (\sigma_2 \sigma_1 \sigma_2)^{-1} $, may be cancelled out via successive isotopies. More precisely, suppose that $ (b_1, \dots, b_k) $ and $ (c_1, \dots, c_l) $ are linear duals. Let $ \mathbf{b'} = (b_1 + x_l, b_2, \dots, b_k + x_r) $ and $ \mathbf{c'} = (c_l + y_l, c_{l-1}, \dots, c_1 + y_r) $ with $ x_i, y_i \geqslant 0 $ for $ i \in \{ l, r \} $ and suppose that $ \mathbf{a} = \mathbf{b'} \vert \mathbf{t} \vert \mathbf{c'} \vert \mathbf{s} $, where $ \mathbf{t} $ and $ \mathbf{s} $ are arbitrary strings, the length of $ \mathbf{t} $ is $ t \geqslant 0 $, and $ \mid $ denotes string concatenation. Consider the sub-braid~$ B $ in the link diagram of $ \widehat{\beta} $ that exactly contains all crossings along the boundary of black regions $ 2, \dots, k-1 $, all but $ x_l + 1 $ leftmost crossings along the boundary of region $ 1 $, and all but $ x_r + 1 $ rightmost crossings along the boundary of region $ k $. Consider also the sub-braid $ C $ that exactly contains all crossings along the boundary of regions $ k + t + 2, \dots, k + t + l - 1 $, all but $ y_l + 1 $ leftmost crossings along the boundary of region $ k + t + 1 $, and all but $ y_r + 1 $ rightmost crossings along the boundary of region $ k + t + l $. Then $ B (\sigma_2 \sigma_1 \sigma_2)^{-1} C = (\sigma_2 \sigma_1 \sigma_2)^{-1}$. Hence, if after applying a band move to $ \widehat{\beta} $ away from $ B $ and $ C $, they are connected by a half-twist of the three strands, one may remove all crossings in $ B $ and $ C $ via isotopies illustrated in Figure~\ref{fig:untongue}. We call $ B $ and $ C $ \emph{dual sub-braids} and enclose them in all following figures in blue and chartreuse rectangles, respectively.

	\begin{proof}[Proof of Proposition~\ref{prop:ribbons}]
		See~\Cref{fig:S2a,fig:S2b,fig:S2d,fig:S2e}.
	\end{proof}

	\begin{center}
		\begin{figure}[h]
			\centering
			\vspace{2em}
			\def\svgwidth{0.45\textwidth}
\begingroup%
  \makeatletter%
  \providecommand\color[2][]{%
    \errmessage{(Inkscape) Color is used for the text in Inkscape, but the package 'color.sty' is not loaded}%
    \renewcommand\color[2][]{}%
  }%
  \providecommand\transparent[1]{%
    \errmessage{(Inkscape) Transparency is used (non-zero) for the text in Inkscape, but the package 'transparent.sty' is not loaded}%
    \renewcommand\transparent[1]{}%
  }%
  \providecommand\rotatebox[2]{#2}%
  \newcommand*\fsize{\dimexpr\f@size pt\relax}%
  \newcommand*\lineheight[1]{\fontsize{\fsize}{#1\fsize}\selectfont}%
  \ifx\svgwidth\undefined%
    \setlength{\unitlength}{157.76516411bp}%
    \ifx\svgscale\undefined%
      \relax%
    \else%
      \setlength{\unitlength}{\unitlength * \real{\svgscale}}%
    \fi%
  \else%
    \setlength{\unitlength}{\svgwidth}%
  \fi%
  \global\let\svgwidth\undefined%
  \global\let\svgscale\undefined%
  \makeatother%
  \begin{picture}(1,0.45637451)%
    \lineheight{1}%
    \setlength\tabcolsep{0pt}%
    \put(0,0){\includegraphics[width=\unitlength,page=1]{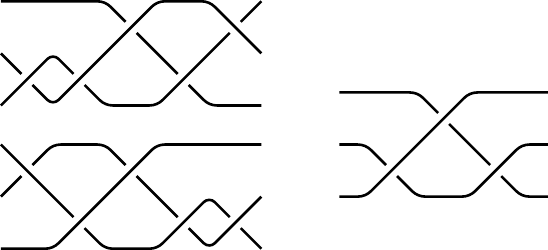}}%
    \put(0.50701762,0.34034166){\color[rgb]{0,0,0}\rotatebox{-45}{\makebox(0,0)[lt]{\lineheight{1.25}\smash{\begin{tabular}[t]{l}$\ra$\end{tabular}}}}}%
    \put(0.52031265,0.08826232){\color[rgb]{0,0,0}\rotatebox{45}{\makebox(0,0)[lt]{\lineheight{1.25}\smash{\begin{tabular}[t]{l}$\ra$\end{tabular}}}}}%
  \end{picture}%
\endgroup%

			\caption{
				Undoing flyped tongues~\cite{Tsukamoto2006} to cancel dual sub-braids.
			} \label{fig:untongue}
		\end{figure}
	\end{center}

	\begin{center}
		\begin{figure}[h]
			\centering
			\def\svgwidth{0.8\textwidth}
\begingroup%
  \makeatletter%
  \providecommand\color[2][]{%
    \errmessage{(Inkscape) Color is used for the text in Inkscape, but the package 'color.sty' is not loaded}%
    \renewcommand\color[2][]{}%
  }%
  \providecommand\transparent[1]{%
    \errmessage{(Inkscape) Transparency is used (non-zero) for the text in Inkscape, but the package 'transparent.sty' is not loaded}%
    \renewcommand\transparent[1]{}%
  }%
  \providecommand\rotatebox[2]{#2}%
  \newcommand*\fsize{\dimexpr\f@size pt\relax}%
  \newcommand*\lineheight[1]{\fontsize{\fsize}{#1\fsize}\selectfont}%
  \ifx\svgwidth\undefined%
    \setlength{\unitlength}{292.87507065bp}%
    \ifx\svgscale\undefined%
      \relax%
    \else%
      \setlength{\unitlength}{\unitlength * \real{\svgscale}}%
    \fi%
  \else%
    \setlength{\unitlength}{\svgwidth}%
  \fi%
  \global\let\svgwidth\undefined%
  \global\let\svgscale\undefined%
  \makeatother%
  \begin{picture}(1,0.15620995)%
    \lineheight{1}%
    \setlength\tabcolsep{0pt}%
    \put(0,0){\includegraphics[width=\unitlength,page=1]{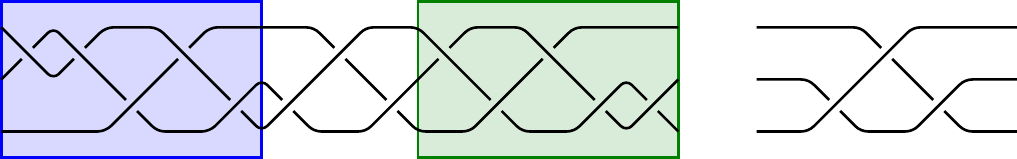}}%
    \put(0.68848529,0.07554416){\color[rgb]{0,0,0}\makebox(0,0)[lt]{\lineheight{1.25}\smash{\begin{tabular}[t]{l}$\ra$\end{tabular}}}}%
  \end{picture}%
\endgroup%

			\caption{
				Cancellation of dual sub-braids for $ (b_1, \dots, b_k) = (2, 2, 3, 3) $ and $ (c_l, \dots, c_1) = (2, 3, 4) $ with $ x_l = x_r = y_l = y_r = 0 $. Fixing the ends on the braid shown, one may remove all crossings in $ B $ and $ C $ via moves illustrated in Figure~\ref{fig:untongue}.
			} \label{fig:cancel-duals}
		\end{figure}
	\end{center}

	\begin{center}
		\begin{figure}[h]
			\centering
			\def\svgwidth{1.1\textwidth}
			\makebox[\textwidth][c]{
\begingroup%
  \makeatletter%
  \providecommand\color[2][]{%
    \errmessage{(Inkscape) Color is used for the text in Inkscape, but the package 'color.sty' is not loaded}%
    \renewcommand\color[2][]{}%
  }%
  \providecommand\transparent[1]{%
    \errmessage{(Inkscape) Transparency is used (non-zero) for the text in Inkscape, but the package 'transparent.sty' is not loaded}%
    \renewcommand\transparent[1]{}%
  }%
  \providecommand\rotatebox[2]{#2}%
  \newcommand*\fsize{\dimexpr\f@size pt\relax}%
  \newcommand*\lineheight[1]{\fontsize{\fsize}{#1\fsize}\selectfont}%
  \ifx\svgwidth\undefined%
    \setlength{\unitlength}{450bp}%
    \ifx\svgscale\undefined%
      \relax%
    \else%
      \setlength{\unitlength}{\unitlength * \real{\svgscale}}%
    \fi%
  \else%
    \setlength{\unitlength}{\svgwidth}%
  \fi%
  \global\let\svgwidth\undefined%
  \global\let\svgscale\undefined%
  \makeatother%
  \begin{picture}(1,0.35166669)%
    \lineheight{1}%
    \setlength\tabcolsep{0pt}%
    \put(0,0){\includegraphics[width=\unitlength,page=1]{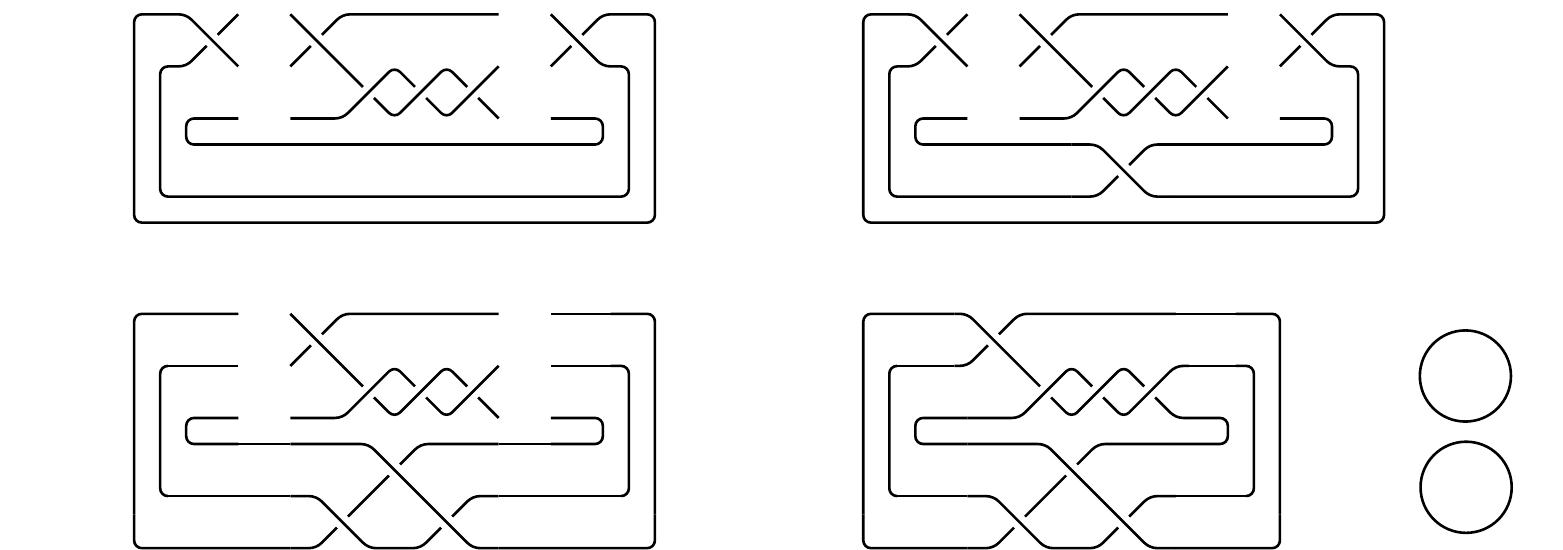}}%
    \put(0.44249238,0.25916669){\color[rgb]{0,0,0}\makebox(0,0)[lt]{\lineheight{1.25}\smash{\begin{tabular}[t]{l}$\xxra{0}[(1)]{\textrm{expand band}}$\end{tabular}}}}%
    \put(-0.00084093,0.07249999){\color[rgb]{0,0,0}\makebox(0,0)[lt]{\lineheight{1.25}\smash{\begin{tabular}[t]{l}$\xxra{0}[(2)]{\textrm{isotopy}}$\end{tabular}}}}%
    \put(0.44582571,0.07249999){\color[rgb]{0,0,0}\makebox(0,0)[lt]{\lineheight{1.25}\smash{\begin{tabular}[t]{l}$\xxra{0}[(3)]{\textrm{cancel duals}}$\end{tabular}}}}%
    \put(0.83915914,0.07250004){\color[rgb]{0,0,0}\makebox(0,0)[lt]{\lineheight{1.25}\smash{\begin{tabular}[t]{l}$\xxra{0}[(4)]{\textrm{isotopy}}$\end{tabular}}}}%
    \put(0,0){\includegraphics[width=\unitlength,page=2]{S2a.pdf}}%
    \put(0.22249241,0.23770185){\color[rgb]{0,0,0}\makebox(0,0)[lt]{\lineheight{1.25}\smash{\begin{tabular}[t]{l}$-1$\end{tabular}}}}%
  \end{picture}%
\endgroup%
}
			\caption{
				Band move for the $ \mathcal{S}_{2a} $ case.
			} \label{fig:S2a}
		\end{figure}
	\end{center}

	\afterpage{\clearpage}
		\begin{figure}[p]
			\centering
			\def\svgwidth{1.1\textwidth}
			\makebox[\textwidth][c]{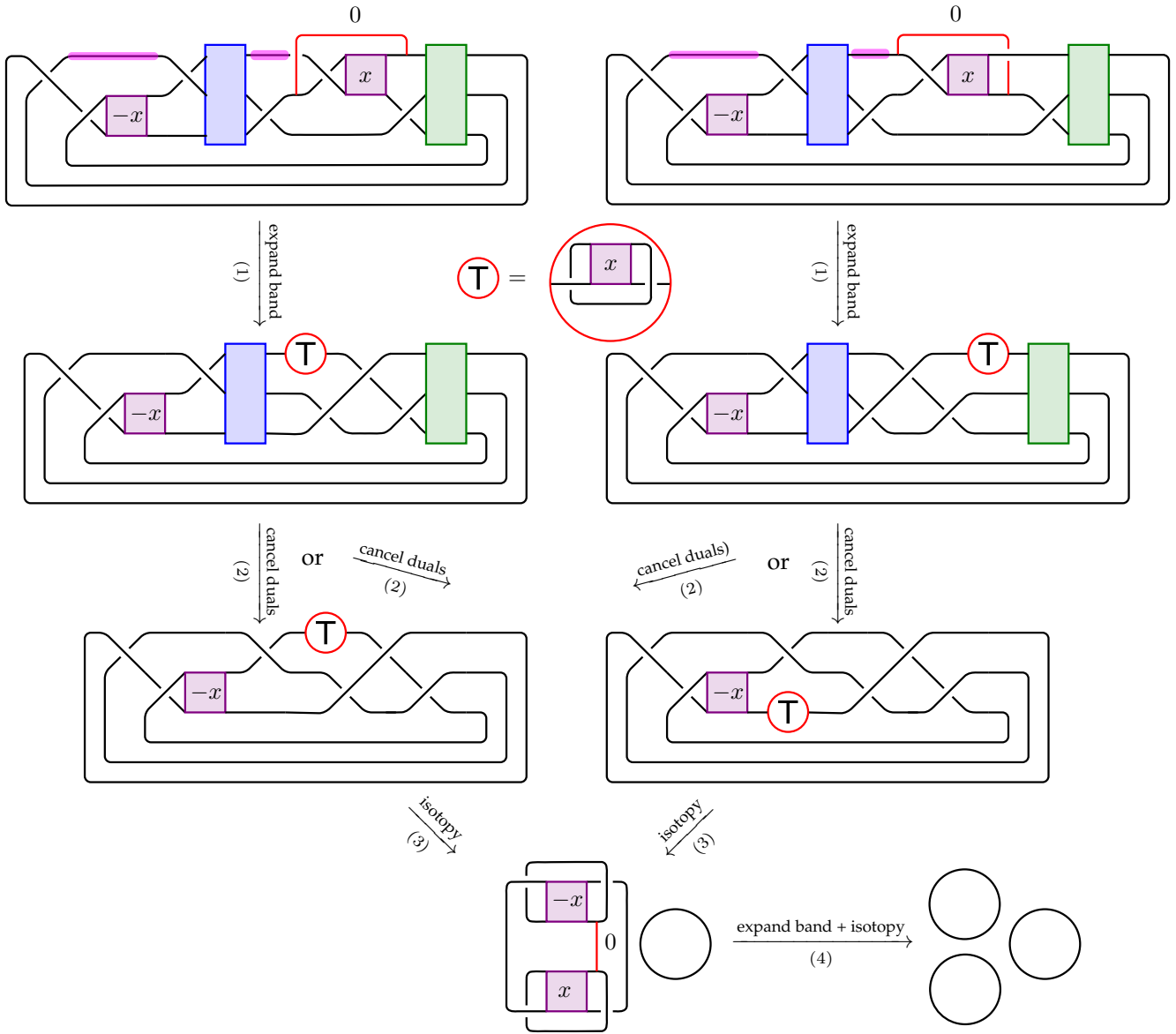}
			\caption{
				Band moves for the $ \mathcal{S}_{2b} $ case. Start with the top left diagram if the two segments highlighted in purple do not lie on the same strand, otherwise start with the top right; this ensures that after step (2), the tangle $ \mathsf{T} $ does not lie on the otherwise unknotted split component. The nontrivial component of the link obtained after step (3) is the connected sum $ T(2, x + 2)\, \#\, T(2, - (x + 2)) $ of two torus links.
			} \label{fig:S2b}
		\end{figure}

	\afterpage{\clearpage}
		\begin{figure}[p]
			\centering
			\def\svgwidth{1.1\textwidth}
			\makebox[\textwidth][c]{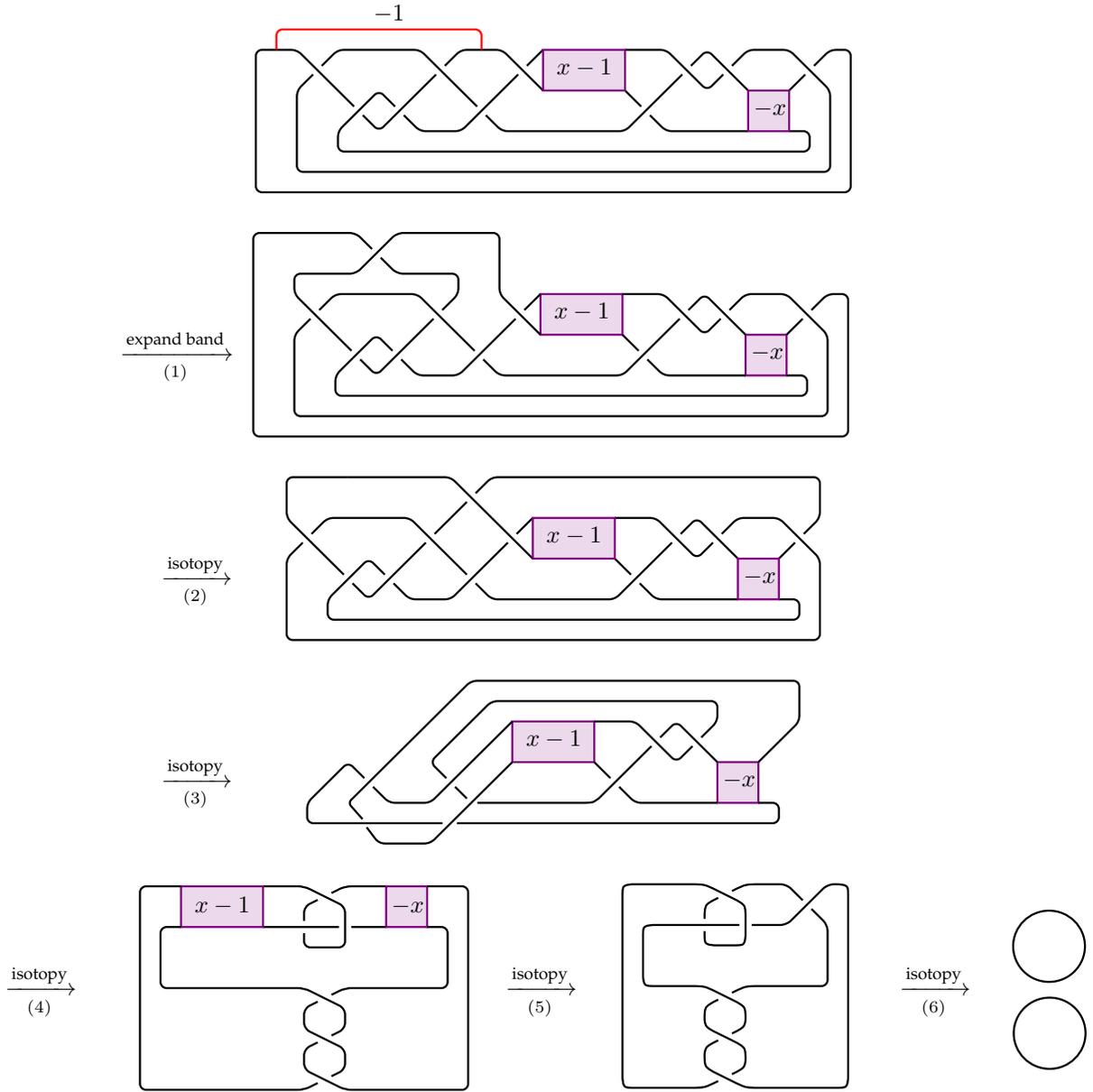}
			\caption{
				Band moves for the $ \mathcal{S}_{2d} $ case with $ x \geqslant 1 $. In step (5), we undo $ x - 1 $ crossings in both blocks by flyping the tangle on the bottom of the diagram and performing Reidemeister II moves. A similar band gives the two-component unlink for the alternating 3-braid closure with associated string $ (2, 2, 2, 4, 4) $.
			} \label{fig:S2d}
		\end{figure}

	\newpage
	\begin{center}
		\begin{figure}[h]
			\centering
			\def\svgwidth{1.1\textwidth}
			\makebox[\textwidth][c]{
\begingroup%
  \makeatletter%
  \providecommand\color[2][]{%
    \errmessage{(Inkscape) Color is used for the text in Inkscape, but the package 'color.sty' is not loaded}%
    \renewcommand\color[2][]{}%
  }%
  \providecommand\transparent[1]{%
    \errmessage{(Inkscape) Transparency is used (non-zero) for the text in Inkscape, but the package 'transparent.sty' is not loaded}%
    \renewcommand\transparent[1]{}%
  }%
  \providecommand\rotatebox[2]{#2}%
  \newcommand*\fsize{\dimexpr\f@size pt\relax}%
  \newcommand*\lineheight[1]{\fontsize{\fsize}{#1\fsize}\selectfont}%
  \ifx\svgwidth\undefined%
    \setlength{\unitlength}{391.12941472bp}%
    \ifx\svgscale\undefined%
      \relax%
    \else%
      \setlength{\unitlength}{\unitlength * \real{\svgscale}}%
    \fi%
  \else%
    \setlength{\unitlength}{\svgwidth}%
  \fi%
  \global\let\svgwidth\undefined%
  \global\let\svgscale\undefined%
  \makeatother%
  \begin{picture}(1,0.38949001)%
    \lineheight{1}%
    \setlength\tabcolsep{0pt}%
    \put(0,0){\includegraphics[width=\unitlength,page=1]{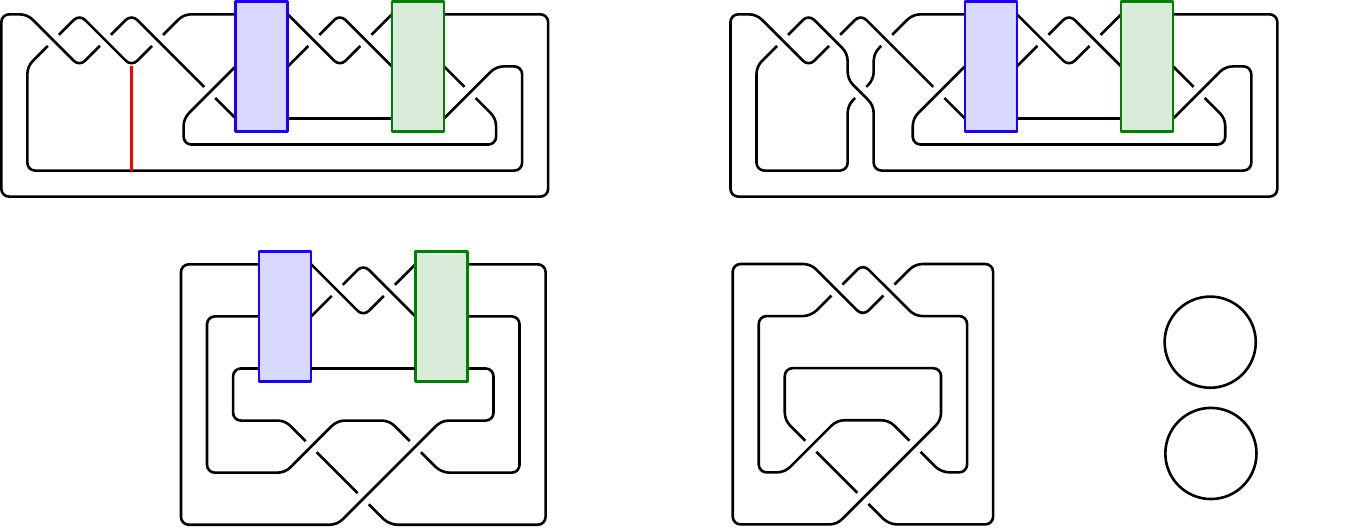}}%
    \put(0.06615458,0.30416023){\makebox(0,0)[lt]{\lineheight{1.25}\smash{\begin{tabular}[t]{l}$-1$\end{tabular}}}}%
    \put(0.4218466,0.30966305){\color[rgb]{0,0,0}\makebox(0,0)[lt]{\lineheight{1.25}\smash{\begin{tabular}[t]{l}$\xxra{0}[(1)]{\textrm{expand band}}$\end{tabular}}}}%
    \put(0.01916654,0.09298297){\color[rgb]{0,0,0}\makebox(0,0)[lt]{\lineheight{1.25}\smash{\begin{tabular}[t]{l}$\xxra{0}[(2)]{\textrm{isotopy}}$\end{tabular}}}}%
    \put(0.4218466,0.09298297){\color[rgb]{0,0,0}\makebox(0,0)[lt]{\lineheight{1.25}\smash{\begin{tabular}[t]{l}$\xxra{0}[(3)]{\textrm{cancel duals}}$\end{tabular}}}}%
    \put(0.75933063,0.09298297){\color[rgb]{0,0,0}\makebox(0,0)[lt]{\lineheight{1.25}\smash{\begin{tabular}[t]{l}$\xxra{0}[(4)]{\textrm{isotopy}}$\end{tabular}}}}%
  \end{picture}%
\endgroup%
}
			\caption{
				Band move for the $ \mathcal{S}_{2e} $ case. A similar band move gives the two-component unlink for the alternating 3-braid closure with associated string $ (2,2,2,3) $.
			} \label{fig:S2e}
			\vspace{-2em}
		\end{figure}
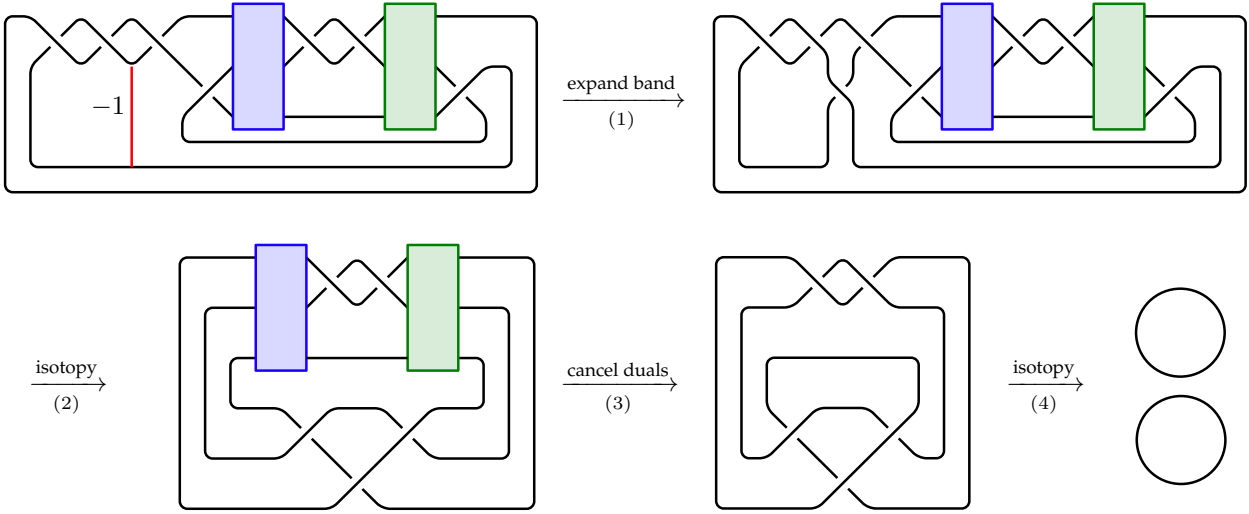
	\end{center}

	In searching for the band moves in~\Cref{fig:S2a,fig:S2b,fig:S2d,fig:S2e}, we have used the algorithm of Owens and Swenton implemented in the \textsf{KLO} program~\cite{Owens2020ribbonalg}. The band moves we exhibit for these four families of alternating 3-braid closures are \emph{algorithmic} in the sense of~\cite{Owens2020ribbonalg}.

	\section{The case of $\mathcal{S}_{2c} \setminus (\mathcal{S}_{2a} \cup \mathcal{S}_{2b} \cup \mathcal{S}_{2d} \cup \mathcal{S}_{2e}) $}
	\label{sec:S2c}

	The remaining $ \mathcal{S}_{2c} $ family is of special interest because it contains strings associated to known examples of non-slice, non-zero determinant alternating 3-braid closures, specifically Turk's head knots $ K_7 $~\cite{sartori2010}, $ K_{11} $, $ K_{17} $ and $ K_{23} $~\cite{ammmps2020branched}; the associated string of $ K_i $ for $ i \in \{ 7, 11, 17, 23 \} $ is $ (3^{[i]}) $. Thus, we should not expect to find a set of band moves for all links with strings in $ \mathcal{S}_{2c} $. We also note that knots of finite concordance order belonging to Family (3) in~\cite{lisca-3braids} have associated strings in $ \mathcal{S}_{2c} $.

	We have that $ \mathcal{S}_{2c} \cap \mathcal{S}_{2d} = \mathcal{S}_{2c} \cap \mathcal{S}_{2e} = \varnothing $: this can be seen by computing the $ I(\mathbf{a}) = \sum_{a \in \mathbf{a}} 3 - a $ invariant~\cite{lisca-rational} which is 0 for strings in $ \mathcal{S}_{2c} $, but 1 or 3 for strings in $ \mathcal{S}_{2d} $ or $ \mathcal{S}_{2e} $, respectively.\footnote{Observe that if $ \mathbf{b} = (b_1, \dots, b_k) $ and $ \mathbf{c} = (c_1, \dots, c_l) $ are linearly dual to each other and $ k + l \geqslant 2 $, then $ I(\mathbf{b} \vert \mathbf{c}) = 0 $.} However, $ \mathcal{S}_{2c} $ has nonzero intersection with $ \mathcal{S}_{2a} $ and $ \mathcal{S}_{2b} $: if one defines a \emph{palindrome} to be a string $(a_1, \dots, a_n)$ such that $ a_i = a_{n-(i-1)} $ for all $1 \leqslant i \leqslant n $, then the following lemma holds.

	\begin{lem}[\protect{\cite[Lemma 3.6]{simone2020classification}}]
	\label{lem:palindromes}
		Let $\mathbf{a}=(b_1+3,b_2,\dots,b_k,2,c_l,\dots,c_1)\in\mathcal{S}_{2a}$ and $\mathbf{b}=(3+x,b_1,\dots,b_{k-1},b_k+1,2^{[x]},c_l+1,c_{l-1},\dots,c_1)\in\mathcal{S}_{2b}.$ Then $\mathbf{a}\in \mathcal{S}_{2c}$ if and only if $(b_1+1,b_2,\dots,b_k)$ is a palindrome and $\mathbf{b}\in\mathcal{S}_{2c}$ if and only if $(b_1\dots,b_k)$ is a palindrome.
	\end{lem}

	We seek to find an easier description of the complement $\mathcal{S}_{2c}^\dagger := \mathcal{S}_{2c} \setminus (\mathcal{S}_{2a} \cup \mathcal{S}_{2b} \cup \mathcal{S}_{2d} \cup \mathcal{S}_{2e}) $. Let
		\[
		\tag{$*$}
			\mathbf{c} = (3+x_1,2^{[x_2]},3+x_3,2^{[x_4]},\dots,3+x_{2k+1},2^{[x_1]}, 3+x_2,2^{[x_3]},\dots,3+x_{2k},2^{[x_{2k+1}]}) \in \mathcal{S}_{2c},
		\]
	where $ k \geqslant 0 $ and $ x_i \geqslant 0 $ for all $ i $. One can more compactly describe $ \mathbf{c} $ by its \emph{$ \mathbf{x} $-string} $ \mathbf{x}(\mathbf{c}) = [x_1, \dots, x_{2k+1}] $ (we use square brackets to denote $ \mathbf{x} $-strings and, as with associated strings, consider them up to cyclic rotations and reversals). For example, the $ \mathbf{x} $-string of $ (3^{[i]}) $ associated with $ K_i $ is $ [0^{[i]}] $. Also, when writing $ \mathbf{c} $ in the form $ (*) $ with the first element being at least 3, call every maximal substring of the form $ (2^{[x]}) $ or $ (3 + x) $ for $ x \geqslant 0 $ an \emph{entry}; the total number of entries $ e(\mathbf{c}) $ in $ \mathbf{c} $ is congruent to $ 2 $ mod $ 4 $.

\begin{lem}
	\label{lem:x-seq}
		Let $\mathbf{a}=(b_1+3,b_2,\ldots,b_k,2,c_l,\ldots,c_1)\in\mathcal{S}_{2a}\cap\mathcal{S}_{2c}$ and $\mathbf{b}=(3+y,b_1,\ldots,b_{k-1},b_k+1,2^{[y]},c_l+1,c_{l-1},\ldots,c_1)\in\mathcal{S}_{2b}\cap\mathcal{S}_{2c}.$ Then
			\begin{myitemize}
				\item $ \mathbf{x}(\mathbf{a}) = [z_1] $ with $ z_1 \geqslant 1 $ or $ \mathbf{x}(\mathbf{a}) = [z_1, \dots, z_{\left \lfloor{\frac{n}{2}} \right \rfloor}, z_{\left \lfloor{\frac{n}{2}} \right \rfloor + 1}, z_{\left \lfloor{\frac{n}{2}} \right \rfloor}, \dots, z_2, z_1 - 2] $ with $ z_1 \geqslant 2 $ and $ n \geqslant 3 $ odd;
				\item $ \mathbf{x}(\mathbf{b}) = [y, 0, z_2] $ or $ \mathbf{x}(\mathbf{b}) = [y, 0, z_2, z_3, \dots, z_{\frac{n}{2}}, z_{\frac{n}{2} + 1}, z_{\frac{n}{2}}, \dots, z_3, z_2 + 1] $ with $ n \geqslant 4 $ even.
			\end{myitemize}
		
	\end{lem}
	\begin{proof}
		Consider $ \mathbf{a} $ and define $ \mathbf{a_c} = (2, c_l, \dots, c_1) $. Notice that $ \mathbf{a_c} $ is the linear dual of the string 
			\[
				\mathbf{a_c^*} = (b_k + 1, b_{k-1}, \dots, b_1),
			\]
		which by Lemma~\ref{lem:palindromes} must be a palindrome, and that
			$
				\mathbf{a} = (b_1 + 3, b_2, \dots, b_k \mid \mathbf{a_c}).
			$
		If $ (b_1, \dots, b_k) $ is the empty string, then $ \mathbf{a} = (2, 1) \notin \mathcal{S}_{2c} $. Otherwise, write
			\[
				\mathbf{a_c} = (2^{[z_1]}, 3 + z_2, \dots, 2^{[z_n]})
			\]
		for $ n \geqslant 1 $ odd and $ z_1 \geqslant 1 $. If $ n = 1 $, then $ \mathbf{a_c} = (2^{[z_1]}) $ and $ \mathbf{a} = (3 + z_1, 2^{[z_1]}) $, so $ \mathbf{x}(\mathbf{a}) = [z_1] $. If $ n > 1 $, then
			\[
			\tag{$**$}
				\mathbf{a_c^*} = (2 + z_1, 2^{[z_2]}, 3 + z_3, \dots, 2^{[z_{n-1}]}, 2 + z_n).
			\]
		Thus,
			\[
				\mathbf{a} = (3 + (z_n + 2), 2^{[z_{n-1}]}, \dots, 2^{[z_2]}, 1 + z_1, 2^{[z_1]}, 3 + z_2, \dots, 2^{[z_n]}).
			\]
		If $ z_1 = 1 $, then
			\[
				\mathbf{a} = (3 + (z_n + 2), 2^{[z_{n-1}]}, \dots, 3 + z_3, 2^{[z_1 + z_2 + 1]}, 3 + z_2, \dots, 2^{[z_n]})
			\]
		does not belong to $ \mathcal{S}_{2c} $ because $ e(\mathbf{a}) \equiv 0 $ mod $ 4 $. If $ z_1 > 1 $, then
			\[
				\mathbf{a} = (3 + (z_n + 2), 2^{[z_{n-1}]}, \dots, 2^{[z_2]}, 3 + (z_1 - 2), 2^{[z_1]}, 3 + z_2, \dots, 2^{[z_n]}) .
			\]
		Now, by considering $ (**) $ we see that $ \mathbf{a_c^*} $ is a palindrome if and only if
			\[
				z_1 = z_n + 2, \quad z_2 = z_{n-1}, \quad \dots, \quad z_{\left \lfloor{\frac{n}{2}} \right \rfloor} = z_{\left \lfloor{\frac{n}{2}} \right \rfloor + 2}
			\]
		so we conclude that $ \mathbf{a} \in \mathcal{S}_{2a} \cap \mathcal{S}_{2c} $ if and only if $ \mathbf{x}(\mathbf{a}) = [z_1] $ for $ z_1 \geqslant 1 $ or $$ \mathbf{x}(\mathbf{a}) = [z_1, z_2, \dots, z_{\left \lfloor{\frac{n}{2}} \right \rfloor}, z_{\left \lfloor{\frac{n}{2}} \right \rfloor + 1}, z_{\left \lfloor{\frac{n}{2}} \right \rfloor}, \dots, z_2, z_1 - 2] $$ for $ z_1 \geqslant 2 $ and $ n \geqslant 3 $ odd.

		Similarly, if $ (b_1, \dots, b_k) $ is empty, then $ \mathbf{b} = (3 + y, 2^{[y]}, 2) = (3 + y, 2^{[y + 1]}) \notin \mathcal{S}_{2c} $. If $ k = 1 $, then the linear dual of $ (b_1) $ with $ b_1 \geqslant 2 $ is $ (2^{[b_1 - 1]}) $, so
		\begin{align*}
			\mathbf{b} = (3 + y, 2^{[0]}, b_1 + 1, 2^{[y]}, 3 + 0, 2^{[b_1 - 2]})
			= (3 + y, 2^{[0]}, 3 + (b_1 - 2), 2^{[y]}, 3 + 0, 2^{[b_1 - 2]})
		\end{align*}
		is indeed in $ \mathcal{S}_{2c} $ and $ \mathbf{x}(\mathbf{b}) = [y, 0, b_1 - 2] $. If $ k > 1 $, write 
			\[
				(b_1, \dots, b_k) = (2^{[z_1]}, 3 + z_2, \dots, 2^{[z_{n-1}]}, 2 + z_n)
			\]
		for $ n \geqslant 2 $ even and $ z_n \geqslant 1 $; its linear dual is
			\[
				(c_1, \dots, c_l) = (2 + z_1, 2^{[z_2]}, 3 + z_3, \dots, 2^{[z_{n-2}]}, 3 + z_{n-1}, 2^{[z_n]}).
			\]
		When $ n = 2 $, we recover the $ k = 1 $ case above, so suppose $ n > 2 $. Then we have 
			\[
				\mathbf{b} = (3 + y, 2^{[z_1]}, \dots, 2^{[z_{n-1}]}, 3 + z_n, 2^{[y]}, 3 + 0, 2^{[z_{n}-1]}, 3 + z_{n-1}, 2^{[z_{n-2}]}, \dots, 3 + z_3, 2^{[z_2 + 1]}).
			\]
		By comparing this with $ (*) $, we see that $ z_1 $ (which corresponds to $ x_2 $) must be zero, and
			\[
				(b_1, \dots, b_k) = (3 + z_2, 2^{[z_3]}, \dots, 2^{[z_{n-1}]}, 3 + (z_n - 1)).
			\]
		The string $ (b_1, \dots, b_k) $ is thus a palindrome precisely when
			\[
				z_2 = z_n - 1, \quad z_3 = z_{n-1}, \quad \dots, \quad z_{\frac{n}{2}} = z_{\frac{n}{2} + 2},
			\]
		i.e., $ \mathbf{x}(\mathbf{b}) = [y, 0, z_2, z_3, \dots, z_{\frac{n}{2}}, z_{\frac{n}{2} + 1}, z_{\frac{n}{2}}, \dots, z_3, z_2 + 1] $. \qedhere
	\end{proof}

	In particular, we can draw the easy conclusion that if $ \mathbf{x}(\mathbf{c}) $ contains neither two adjacent elements differing by $ 2 $ nor a $ 0 $, then $ \mathbf{c} \in \mathcal{S}_{2c}^\dagger $. We now show that infinitely many $\chi$-ribbon links have their associated strings in $ \mathcal{S}_{2c}^\dagger $.

	\begin{lem}
	\label{lem:S2c_family}
		Let $ \widehat{\beta} $ be the closure of $ \beta = \sigma_1^{m + 1} (\sigma_2^{-1} \sigma_1)^2 \sigma_2^{-(m+1)} (\sigma_1 \sigma_2^{-1})^2 $ with the associated string $ \mathbf{c} = (3 + m, 3, 3, 2^{[m]}, 3, 3) $ and $ m \geqslant 3 $. Then $ \mathbf{c} \in \mathcal{S}_{2c}^\dagger $ and $ \widehat{\beta} $ admits a ribbon surface with a single 1-handle.
	\end{lem}
	\begin{proof}
		We have $ \mathbf{x}(\mathbf{c}) = [m, 0, 0, 0, 0] $, so by Lemma~\ref{lem:x-seq}, $ \mathbf{c} \in \mathcal{S}_{2c}^\dagger $. For the band move, see Figure~\ref{fig:S2c_family}.
	\end{proof}

	\afterpage{\clearpage}
	\begin{figure}[p]
		\centering
		\def\svgwidth{0.96\textwidth}
		\makebox[\textwidth][c]{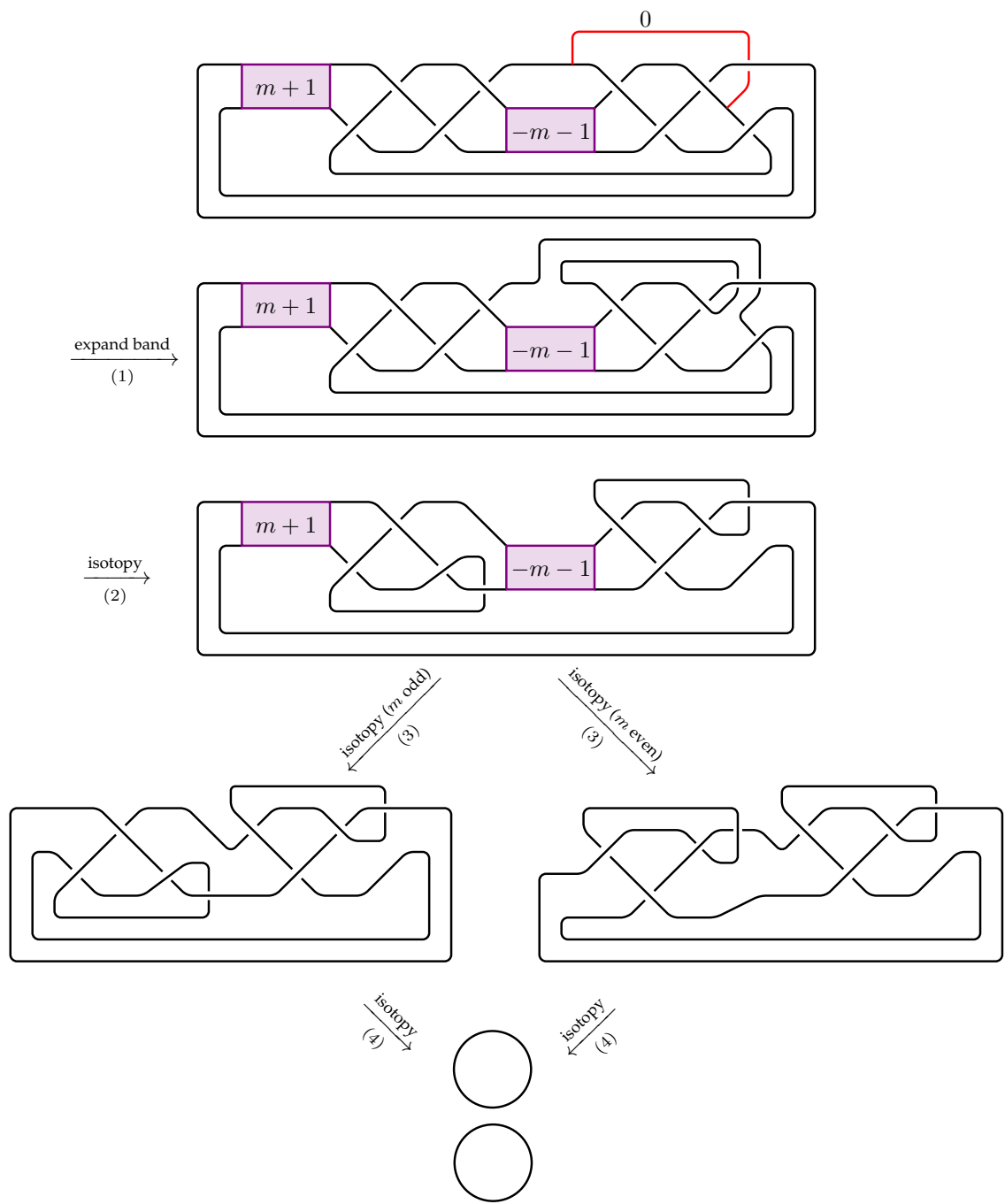}
		\caption{
			Band moves for an alternating 3-braid closure with $ \mathbf{x} $-string $ [m, 0, 0, 0, 0] $ for $ m \geqslant 3 $. In (3), we perform $ m + 1 $ flypes of the tangle between two blocks with $ m $ crossings followed by Reidemeister II moves.
		} \label{fig:S2c_family}
	\end{figure}

	Using \textsf{KLO}, we have found that 22 out of 33 closures of alternating 3-braids with up to 20 crossings whose associated strings belong to $ \mathcal{S}_{2c}^\dagger $ are algorithmically ribbon, in each instance via at most two band moves. It is known that the Turk's head knot $ K_7 $ with the associated string in $ \mathcal{S}_{2c}^\dagger $ and 14 crossings is not slice~\cite{sartori2010}. The remaining 10 examples for which we were unable to find band moves making their diagrams isotopic to the unlink are listed in Table~\ref{tab:S2c-nonalg}. By a straightforward application of the Gordon-Litherland signature formula~\cite[Theorems~6~and~6'']{gordonlitherland}, the signature of the closure of a braid $ \beta = \sigma_1^{a_1} \sigma_2^{-b_1} \dots \sigma_1^{a_n} \sigma_2^{-b_n} $ with $ \sum_i a_i $ and $ \sum_i b_i $ both greater than one is
		\[
			\sigma(\widehat{\beta}) = \sum_{i=1}^n b_i - a_i.
		\]
	Thus, for all links with associated strings in $ \mathcal{S}_{2a} \cup \mathcal{S}_{2b} \cup \mathcal{S}_{2c} $ satisfying this condition (in particular, for those in Table~\ref{tab:S2c-nonalg}), the signature vanishes, which means that for knots, so do the Ozsv\'ath and Sz\'abo's~$ \tau $ and Rasmussen's $ s $ invariants~\cite{ozsz2003,rasmussen2004} without giving us any sliceness obstructions; Tristram-Levine signatures for knots in Table~\ref{tab:S2c-nonalg} are also zero. Moreover, by comparing their hyperbolic volumes, we have verified that none of the entries in Table~\ref{tab:S2c-nonalg} belong to the list of `escapee' $\chi$-ribbon links described in~\cite{Owens2020ribbonalg}: this further advances them as candidates for more careful study. In Section~\ref{sec:tap} we will show that the three knots $ K_1 $, $ K_2 $ and $ K_3 $ in Table~\ref{tab:S2c-nonalg} are not slice, which lets us conclude that every knot which is a closure of an alternating 3-braid with up to 20 crossings and whose double branched cover bounds a rational ball, except $ K_1$, $ K_2 $, $ K_3 $ and $ K_7 $, is slice.

	\begin{center}{}
	\begin{table}[h]
		\vspace{1em}
		\centering
		\begin{tabular}{ccccc}
		\# of crossings & Associated string & $\mathbf{x}$-string & \# of components \\ 
		\hline
		\rule{0pt}{2.5ex}18 & $(3^{[9]})$ & $[0^{[9]}]$ & 3 \\ 
		18 & $(2,4,2,4,4,2,4,2,3)$   & $[1,1,1,1,0]$ & 1 \\ 
		18 & $(2,2,4,3,2,5,2,3,4)$   & $[2,1,0,0,1]$ & 1 \\ 
		18 & $(2,3,4,3,4,3,2,3,3)$   & $[1,0,0,0,1,0,0]$ & 1 \\ 
		20 & $(2,2,2,3,3,3,6,3,3,3)$ & $[3,0^{[6]}]$ & 3 \\ 
		20 & $(2,4,2,4,2,4,2,4,2,4)$ & $[1^{[5]}]$ & 3 \\ 
		20 & $(2,4,2,3,3,4,2,4,3,3)$ & $[1,1,1,0,0,0,0]$ & 3 \\ 
		20 & $(2,4,3,2,3,4,2,3,4,3)$ & $[1,1,0,0,1,0,0]$ & 3 \\ 
		20 & $(2,3,2,3,2,3,4,4,4,3)$ & $[1,0,1,0,1,0,0]$ & 3 \\ 
		20 & $(2,2,2,4,3,2,6,2,3,4)$ & $[3,1,0,0,1]$ & 3 \\ 
		\\
		\end{tabular}
		\caption{Links in $ \mathcal{S}_{2c}^\dagger $ with up to 20 crossings which are potentially non-$ \chi $-slice. In the following we show that the three knots in this table are not slice.\vspace{-1.5em}}
		\label{tab:S2c-nonalg}
	\end{table}
	\end{center}

	\begin{rem}
		We note that not all alternating knots can be represented as closures of alternating braids. This implies that our list of smoothly non-slice knots which are closures of alternating 3-braids with up to 20 crossings does not include, for example, the non-slice alternating knot $ 5_2 $, which has braid index~3, but cannot be represented as a closure of any alternating braid~\cite{cromwell1989}. A full classification of braid presentations of alternating links with braid index 3 has been given by Stoimenow in~\cite{stoimenow2003}.
	\end{rem}

	\section{Three more non-slice knots in $ \mathcal{S}_{2c}^\dagger $}
	\label{sec:tap}

	In this section we restrict our attention to the three knots in Table~\ref{tab:S2c-nonalg}. Let
	\begin{align*}
		\beta_1 &= \sigma_1^{2} \sigma_2^{-2} \sigma_1^{2} \sigma_2^{-2} \sigma_1 \sigma_2^{-2} \sigma_1^{2} \sigma_2^{-2} \sigma_1^{2} \sigma_2^{-1}, \\
		\beta_2 &= \sigma_1^{3} \sigma_2^{-2} \sigma_1 \sigma_2^{-1} \sigma_1^{2} \sigma_2^{-3} \sigma_1^{2} \sigma_2^{-1} \sigma_1 \sigma_2^{-2}, \\
		\beta_3 &=  \sigma_1^{2} \sigma_2^{-1} \sigma_1 \sigma_2^{-2} \sigma_1 \sigma_2^{-1} \sigma_1 \sigma_2^{-2} \sigma_1 \sigma_2^{-1} \sigma_1^{2} \sigma_2^{-1} \sigma_1 \sigma_2^{-1},
	\end{align*}
	and let $ K_i = \widehat{\beta}_i $ for $ i = 1, 2, 3 $. We will show that the knots $ K_i $ are not slice by adapting the approach of Aceto et al.~\cite{ammmps2020branched}, based in turn on work of Herald, Kirk and Livingston~\cite{heraldkirklivingston}, and demonstrating that certain reduced twisted Alexander polynomials do not factor as norms; this is a generalisation of the Fox-Milnor condition on Alexander polynomials of $ K_i $ which is passed by these knots. Fix distinct primes $ p $ and $ q $, and let $ \zeta_q $ denote a $ q^{\textrm{th}} $ root of unity. The general outline of the algorithm is the following:
		\begin{enumerate}[label=\arabic*.,leftmargin=0.1\linewidth]
			\item Construct the Seifert matrix $ S_i $ for $ K_i $ coming from the standard Seifert surface $ F_i $ associated to $ K_i $ viewed as a 3-braid closure.
			\item By considering the presentation matrix $ P_i = t S_i - S_i^T \in \mathrm{Mat}(\Z[t^{\pm 1}]) $ of the Alexander module $ \mathcal{A}(K_i) $, determine the structure of $ H_1(\Sigma_p(K_i)) $, the first homology of the $ p $-fold cover of $ S^3 $ branched over $ K_i $, as well as a basis of $ H_1(\Sigma_p(K_i)) $ given by lifts of curves in $ S^3 \setminus \nu(F) $.
			\item Calculate the Blanchfield pairings $ \textrm{Bl}_i : \mathcal{A}(K_i) \times \mathcal{A}(K_i) \ra \Q(t)/\Z[t^{\pm 1}]$ and deduce the linking pairings $ \lambda_i : H_1(\Sigma_p(K_i)) \times H_1(\Sigma_p(K_i)) \ra \Q / \Z$.
			\item Enumerate all $ \Z[t^{\pm 1}] $-submodules $ N $ of $ H_1(\Sigma_p(K_i)) $ with $ |N|^2 = | H_1(\Sigma_p(K_i)) | $ and thus find all metabolisers of $ H_1(\Sigma_p(K_i)) $, i.e., those $ N $ on which $ \lambda_i $ vanishes.
			\item Construct nontrivial characters $ \chi : H_1(\Sigma_3(K_i)) \ra \Z/q $ that vanish on the metabolisers.
			\item Using a Wirtinger presentation of $ \pi_1(X_i) $, where $ X_i $ is the knot complement of~$ K_i $, construct a certain homomorphism $ \pi_1(X_i) \ra \Z \ltimes H_1(\Sigma_p(K_i)) $ that induces a representation $ \varphi_\chi : \pi_1(X_i) \ra \GL(p, \Q(\zeta_q)[t^{\pm 1}])$ for each character in (5).
			\item Use the Fox matrix for a Wirtinger presentation of $ \pi_1(X_i) $ to obtain a matrix $ \Phi_\chi $ for each $ \chi $ in (5), whose determinant $ \det \Phi_\chi $ is the reduced twisted Alexander polynomial $ \widetilde{\Delta}_{K_i}^\chi(t) $.
			\item Verify that none of the $ \widetilde{\Delta}_{K_i}^\chi(t) $ factor as norms, hence providing an obstruction to sliceness of all $ K_i $.
		\end{enumerate}

	For reference about various terms used in this outline, we direct the reader in the first instance to~\cite{heraldkirklivingston} and \cite{ammmps2020branched}, as well as to the survey~\cite{friedlvidussi}. The computations were performed in \texttt{SageMath} notebooks available on the author's website\footnote{~\url{https://sites.google.com/view/vbrej}}.

	\subsection{The Seifert matrix}
	\label{subsec:seifert} Let $ \beta $ be a 3-braid. A Seifert surface $ F $ for $ \widehat{\beta} $ can be constructed by joining three discs $ D_1 $, $ D_2 $ and $ D_3 $ by half-twisted bands, where each band between $ D_1 $ and $ D_2 $ comes from a $ \sigma_1 $ term in $ \beta $, and each band between $ D_2 $ and $ D_3 $ from a $ \sigma_2 $ term; identify the bands with $ \sigma_i $'s. Let $ g $ be the genus of $ F $. We can choose the generators of $ H_1(F) $ to be the loops running once through consecutive $ \sigma_1 $'s and $ \sigma_2 $'s, except for the loop between the first and last $ \sigma_1 $ and the first and last $ \sigma_2 $. We order these generators $ s_1, \dots, s_{2g} $ by when the first $ \sigma_i $ through which $ s_j $ runs appears in $ \beta $. With this setup, the Seifert matrix $ S $ can be obtained using the algorithm of Collins~\cite{collins2016algorithm}. Such $ F $ with $ s_1, \dots, s_{2g} $ for $ K_1 $ is shown in Figure~\ref{fig:K1_seifert}. Also, for $ \nu(F) $ an open tubular neighbourhood of $ F $, denote by $ \widehat{s}_i $ a choice of a simple closed curve in $ S^3 \setminus \nu(F) $ that is \emph{Alexander dual} to $ \{ s_1, \dots, s_{2g} \} $, i.e., which satisfies $ \mathrm{lk}(s_i, \widehat{s}_j) = \delta_{ij} $.
	\begin{figure}[h]
		\centering
		\def\svgwidth{1.1\textwidth}
		\makebox[\textwidth][c]{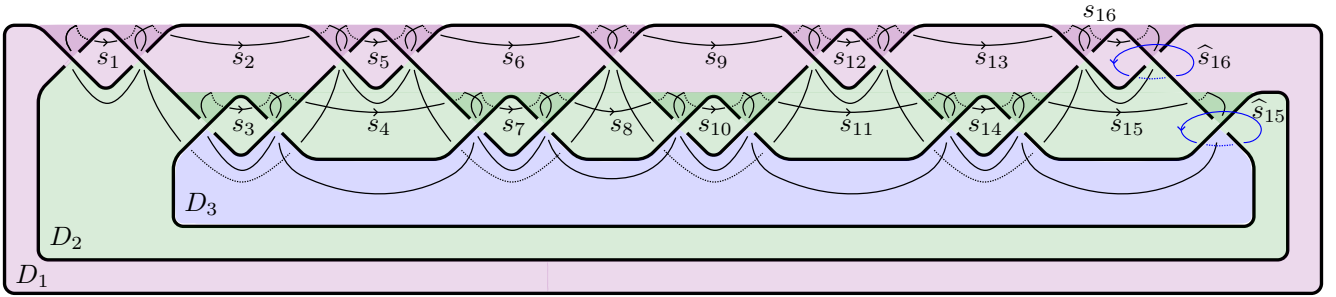}
		\caption{
			Our choice of a Seifert surface $ F_1 $ for $ K_1 $. Lifts of Alexander dual curves $ \widehat{s}_{15} $ and $ \widehat{s}_{16} $ will turn out to generate $ H_1(\Sigma_3(K_1)) $.
		} \label{fig:K1_seifert}
	\end{figure}

	\subsection{Structure and bases of triple branched covers $ H_1(\Sigma_3(K_i)) $} We may perform column operations on the presentation matrices $ P_i = t S_i - S_i^T $ of the Alexander modules $ \mathcal{A}(K_i) $ to transform them into the forms
	\[
		\left(
			\begin{array}{c | c  c}
				I & \multicolumn{2}{c}{0} \\
				\hline
				\multirow{2}{*}{$\ast$} & p_1(t) & 0 \\
				& 0 & p_1(t)
			\end{array}
		\right), \quad
		\left(
			\begin{array}{c | c  c  c}
				I & \multicolumn{3}{c}{0} \\
				\hline
				\multirow{3}{*}{$\ast$} & p_2(t) & 0 & 0 \\
				& 0 & 1 & 0 \\
				& 0 & * & p_2(t)
			\end{array}
		\right), \quad
		\left(
			\begin{array}{c | c  c}
				I & \multicolumn{2}{c}{0} \\
				\hline
				\multirow{2}{*}{$\ast$} & p_3(t) & 0 \\
				& 0 & p_3(t)
			\end{array}
		\right)
	\]
	for $ i = 1, 2, 3 $, respectively, where each $ p_i(t) $ is the square root of the untwisted Alexander polynomial $ \Delta_{K_i}(t) $, $ I $ is the identity matrix and $ * $ represents other entries. Specifically,
	\begin{align*}
		p_1(t) &= 1 - 3t + 7t^2 - 10t^3 + 11t^4 - 10t^5 + 7t^6 - 3t^7 + t^8, \\
		p_2(t) &= 1 - 3t + 6t^2 - 9t^3 + 11t^4 - 9t^5 + 6t^6 - 3t^7 + t^8, \\
		p_3(t) &= 1 - 4t + 8t^2 - 11t^3 + 13t^4 - 11t^5 + 8t^6 - 4t^7 + t^8.
	\end{align*}
	Recall that the Alexander module $ \mathcal{A}(K) $ of a knot $ K $ is the $ \Z[t^{\pm 1}] $-module $ H_1(\widetilde{X^\infty_K}) $, where $ \widetilde{X^\infty_K} $ is the infinite cyclic cover of the knot complement $ X_K $ and $ t $ acts by deck transformations. Choose a preferred copy of $ S^3 \setminus \nu(F_i) $ in $ \widetilde{X^\infty_{K_i}} $ for all $ i $. From~\cite[Theorems 1.3 and 1.4]{friedl2017calculation}, summarised in the present context in~\cite[Theorem 3.6]{ammmps2020branched}, it follows that
	\[
		\mathcal{A}(K_i) \cong \Z[t^{\pm 1}]/\langle p_i(t) \rangle \oplus \Z[t^{\pm 1}]/\langle p_i(t) \rangle,
	\]
	where $ \mathcal{A}(K_i) $ for $ i \in \{ 1, 3 \} $ is generated by the lifts of $ \widehat{s}_{15} $ and $ \widehat{s}_{16} $ to the preferred copy of $ S^3 \setminus \nu(F_i) $ in $ \widetilde{X^\infty_i} $, while $ \mathcal{A}(K_2) $ is generated by the lifts of $ \widehat{s}_{14} $ and $ \widehat{s}_{16} $; in each case, call these generators $ a $ and $ b $, respectively. Choose $ p = 3 $. By, e.g.,~\cite[Section~6.1]{feller2019balanced}, we have
	\begin{align*}
		H_1(\Sigma_3(K_i)) &\cong \mathcal{A}(K_i) / \langle t^2 + t + 1 \rangle \\
		&\cong \Z[t^{\pm 1}]/\langle p_i(t), t^2 + t + 1 \rangle \oplus \Z[t^{\pm 1}]/\langle p_i(t), t^2 + t + 1 \rangle \\
		&\cong \Z[t^{\pm 1}]/ \langle 7t, t^2 + t + 1 \rangle \oplus \Z[t^{\pm 1}]/ \langle 7t, t^2 + t + 1 \rangle \\
		&\cong (\Z/7)[t^{\pm 1}]/\langle t^2 + t + 1 \rangle \oplus (\Z/7)[t^{\pm 1}]/\langle t^2 + t + 1 \rangle
	\end{align*}
	in each of the three cases, since all of $ p_i(t) $ are congruent to $ 7t $ modulo $ t^2 + t + 1 $. Hence, we fix $ q = 7 $. The generators of $ \mathcal{A}(K_i) $ descend to $ H_1(\Sigma_3(K_i)) $, so by abuse of notation we also denote them by $ a $ and $ b $. As a group, $ H_1(\Sigma_3(K_i)) \cong (\Z/7)^4 $, and we may treat it as a $ (\Z/7) $-module generated by $ a, ta, b $ and $ tb $.

	\subsection{Blanchfield and linking forms} Following~\cite[Theorems 1.3 and 1.4]{friedl2017calculation},~\cite[Theorem 3.6]{ammmps2020branched} and a calculation in the accompanying notebooks, we obtain that the Blanchfield pairings on $ \mathcal{A}(K_i) $ are given, with respect to the ordered basis $ \{a, b\} $ and after reducing both the numerators and denominators modulo $ t^3 - 1 $, by
	\begin{align*}
		\frac{1}{7}&\begin{pmatrix}
			2t^2+2t-4 &  -2t^2+4t-2\\ 
			4t^2-2t-2 & -4t^2-4t+8
		\end{pmatrix}, \\[5pt]
		\frac{1}{7}&\begin{pmatrix}
			-3t^2-3t+6 &  3t^2-3t\\ 
			-3t^2+3t & 3t^2+3t-6
		\end{pmatrix}, \\[5pt]
		\frac{1}{7}&\begin{pmatrix}
			-4t^2-4t+8 & 4t^2-2t-2\\ 
			-2t^2+4t-2 & 2t^2+2t-4
		\end{pmatrix}
	\end{align*} \\
	for $ i = 1, 2, 3 $, respectively. Via~\cite[Chapter~2.6]{friedlthesis}, applied similarly to~\cite[Proposition~3.7]{ammmps2020branched}, we read off that the linking forms $ \lambda_i : H_1(\Sigma_3(K_i)) \times H_1(\Sigma_3(K_i)) \ra \Q/\Z $  with respect to the ordered basis $ \{a,ta,b,tb\} $ are given by
	\[
		\frac{1}{7}\begin{pmatrix}
		-4 & 2 & -2 & 4 \\
		2 & -4 & -2 & -2 \\
		-2 & -2 & 1 & -4 \\
		4 & -2 & -4 & 1
		\end{pmatrix}, \quad
		\frac{1}{7}\begin{pmatrix}
		6 & -3 & 0 & -3 \\
		-3 & 6 & 3 & 0 \\
		0 & 3 & -6 & 3 \\
		-3 & 0 & 3 & -6
		\end{pmatrix} \quad \textrm{and} \quad
		\frac{1}{7}\begin{pmatrix}
		1 & -4 & -2 & -2 \\
		-4 & 1 & 4 & -2 \\
		-2 & 4 & -4 & 2 \\
		-2 & -2 & 2 & -4
		\end{pmatrix}.
	\]

	\subsection{Metabolisers of $ H_1(\Sigma_3(K_i)) $} Write $ M = (\Z/7) [t^{\pm 1}] / \langle t^2 + t + 1 \rangle $ so that, as a $(\Z/7) [t^{\pm 1}]$-module, $ H_1(\Sigma_3(K_i)) \cong M \oplus M $. Since the order $ |H_1(\Sigma_3(K_i))| = 7^4 $, we seek to describe all its $ \Z[t^{\pm 1}] $-submodules of order $ 7^2 = 49 $. Since $ t^2 + t + 1 $ has irreducible factors $ (t - 2),\,(t + 3) \in (\Z/7)[t^{\pm 1}]$, the set $ \{ \langle 0 \rangle, \langle 1 \rangle, \langle t - 2 \rangle, \langle t + 3 \rangle \} $ contains precisely the $ (\Z/7)[t^{\pm 1}] $-submodules of $ M $; since the $ \Z[t^{\pm 1}] $-action on $ M $ factors through $ (\Z/7)[t^{\pm 1}] $, these are also precisely the $ \Z[t^{\pm 1}] $-submodules of $ M $. Observe that $ |\langle 0 \rangle| = 1 $, $ |\langle 1 \rangle| = 49 $ and $ |\langle t - 2 \rangle| = |\langle t + 3 \rangle| = 7 $. Now let $ N $ be a  $ \Z[t^{\pm 1}] $-submodule of $ H_1(\Sigma_3(K_i)) $,  and consider the commutative diagram
	\[
		\xymatrix{
			M \oplus \{ 0 \} \ar[r] & M \oplus M \ar[r]^>>>>>{\pi} & \{0\} \oplus M \\
			\ker \pi|_N \ar[r] \ar[u] & N \ar[r]^{\pi|_N} \ar[u] & \im\, \pi|_N \ar[u]
		}
	\]
	where $ \pi(x, y) = (0, y) $ for all $ x, y \in M $, and unlabelled arrows are inclusions; $ \ker \pi|_N $ and $ \im\,\pi|_N $ are submodules of $ M \oplus \{ 0 \} $ and $ \{ 0 \} \oplus M $, respectively. Since $ |N| = |\ker \pi|_N|\cdot|\im\,\pi|_N|$, we can deduce what $ N $ could be by order considerations.
		\begin{itemize}
			\item If $ |\ker \pi|_N| = 49 $, then $ |\im\, \pi|_N| = 1 $ and $ N = \ker \pi|_N = \textrm{span}_{(\Z/7)[t^{\pm 1}]}\{ (1, 0) \} $.
			\item If $ |\ker \pi|_N| = 1 $, then $ N \cong \im\, \pi|_N = \textrm{span}_{(\Z/7)[t^{\pm 1}]}\{(k, 1)\} $ for some $ k \in (\Z/7)[t^{\pm 1}] $.
		\end{itemize}
	Now, let $ \{ \langle t - 2 \rangle, \langle t + 3 \rangle \} = \{ \langle \alpha \rangle, \langle \beta \rangle \} $; we have $ \mathrm{Ann}\,\alpha = \langle \beta \rangle $ and $ \mathrm{Ann}\,\beta = \langle \alpha \rangle $. There are two remaining cases to consider.
		\begin{itemize}
			\item Suppose $ \ker \pi|_N \cong \im\, \pi|_N \cong \langle \alpha \rangle $. Then $ N $ contains $ \{ (\alpha, 0), (k, \alpha) \} $ for some $ k \in (\Z/7)[t^{\pm 1}] $. Since $ \beta (k, \alpha) = (\beta k, 0) \in \ker \pi|_N $, we must have $ \beta k \in \langle \alpha \rangle $, so $ k \in \langle \alpha \rangle $, i.e., $ k = l\alpha $ for some $ l \in (\Z/7)[t^{\pm 1}] $. Then $ -l(\alpha, 0) + (k, \alpha) = (0, \alpha) \in N $, so $ N $ contains two linearly independent elements $ (\alpha, 0)$ and $(0, \alpha) $ of order $ 7 $, hence is generated by them for any choice of $ k $. This yields two submodules $ N = \mathrm{span}_{(\Z/7)[t^{\pm 1}]}\{ (t - 2, 0), (0, t - 2) \} $ and $ N = \mathrm{span}_{(\Z/7)[t^{\pm 1}]}\{ (t + 3, 0), (0, t + 3) \} $. 
			\item Suppose $ \ker \pi|_N = \langle \alpha \rangle $ and $ \im\, \pi|_N \cong \langle \beta \rangle $. We similarly observe that $ N $ contains $ \{ (\alpha, 0), (k, \beta) \} $ for some $ k \in (\Z/7)[t^{\pm 1}] $. We have $ \alpha (k, \beta) = (\alpha k, 0) \in \ker \pi|_N $, so we can take $ k $ modulo $ \alpha $, i.e., $ k \in \Z/7 $. Then $ \{ (\alpha, 0), (k, \beta) \} $ is a linearly independent set generating $ N $ for any choice of $ k \in \Z/7 $. Thus, $ N = \mathrm{span}_{(\Z/7)[t^{\pm 1}]}\{ (t - 2, 0), (k, t + 3) \} $ or $ N = \mathrm{span}_{(\Z/7)[t^{\pm 1}]}\{ (t + 3, 0), (k, t - 2) $ for $ k \in \Z/7 $.
		\end{itemize}

	To summarise, writing elements of $ H_1(\Sigma_3(K_i)) \cong M \oplus M $ additively with the first copy of $ M $ generated by $ a $ and the second by $ b $, the desired submodules are
		\begin{align*}
			N_0 &= \mathrm{span}_{(\Z/7)[t^{\pm 1}]} \{ a \}; \\
			N_{k_0,k_1} &= \mathrm{span}_{(\Z/7)[t^{\pm 1}]} \{ ka + b \} \textrm{ for } k \in (\Z/7)[t^{\pm 1}] \\
			&\qquad= \mathrm{span}_{(\Z/7)[t^{\pm 1}]} \{ (k_0 + k_1 t)a + b \} \textrm{ for } k_0, k_1 \in \Z/7; \\
			N_0^\alpha &= \mathrm{span}_{(\Z/7)[t^{\pm 1}]} \{ (t - 2)a, (t - 2)b \}; \\
			N_0^\beta &= \mathrm{span}_{(\Z/7)[t^{\pm 1}]} \{ (t + 3)a, (t + 3)b \}; \\
			N_{k_0}^{\alpha\beta} &= \mathrm{span}_{(\Z/7)[t^{\pm 1}]} \{ (t - 2)a, k_0 a + (t + 3)b \} \textrm{ for } k_0 \in \Z/7; \\
			N_{k_0}^{\beta\alpha} &= \mathrm{span}_{(\Z/7)[t^{\pm 1}]} \{ (t + 3)a, k_0 a + (t - 2)b \} \textrm{ for } k_0 \in \Z/7.
		\end{align*}
	By a direct computation carried out in the accompanying notebooks, the submodules $ N_0^\alpha $ and $ N_0^\beta $ are metabolisers for $ K_i $ for all $ i $; in addition, $ K_1 $ has metabolisers $ N_6^{\alpha \beta} $ and $ N_4^{\beta \alpha} $, $ K_2 $ has metabolisers $ N_1^{\alpha \beta} $ and $ N_1^{\beta \alpha} $, and $ K_3 $ has metabolisers $ N_2^{\alpha \beta} $ and $ N_3^{\beta \alpha} $.

	\subsection{Characters vanishing on the metabolisers}
	\label{subsec:characters} It is easy to define characters $ \chi : H_1(\Sigma_3(K_i)) \ra \Z/7 $ that vanish on the metabolisers. Let subscripts and superscripts denote corresponding metabolisers and 4-tuples in parentheses represent the values a character takes on the ordered basis $ \{ a, ta, b, tb \} $. Then we can take $\chi_0^\alpha $ and $ \chi_0^\beta $ as defined by $ (1, 2, 1, 2) $ and $ (1, -3, 1, -3) $, respectively. The rest of the characters are presented in Table~\ref{tab:characters}.
	\begin{table}[h]
	\begin{tabular}{c c c}
	$K_1$ & $K_2$ & $K_3$ \\[2pt]
	\hline\\[\dimexpr-\normalbaselineskip+2pt]
	$\chi_6^{\alpha\beta} : (1, 2, 1, -2)$  & $\chi_1^{\alpha\beta} : (1, 2, 1, -4)$ & $\chi_2^{\alpha\beta} = \chi_0^\alpha: (1, 2, 1, 2)$ \\[2pt]
	$\chi_4^{\beta\alpha} : (1, -3, 1, -2)$ & $\chi_1^{\beta\alpha} : (1, -3, 1, 1)$ & $\chi_3^{\beta\alpha} : (1, -3, 1, 1)$
	\end{tabular}
		\caption{Our choice of characters $ \chi : H_1(\Sigma_3(K_i)) \ra \Z/7 $ vanishing on the metabolisers of $ K_1 $, $ K_2 $ and $ K_3 $; the characters $\chi_0^\alpha $ and $ \chi_0^\beta $ are given for all $ K_i $ by $ (1, 2, 1, 2) $ and $ (1, -3, 1, -3) $.}
		\label{tab:characters}
	\end{table}

	\subsection{Representations of the knot groups into $ \GL(3, \Q(\zeta_7)[t^{\pm 1}]) $}
	\label{subsec:wirtinger} Let $ K \in \{ K_1, K_2, K_3 \}$. We follow~\cite[Appendix~A]{ammmps2020branched} and~\cite[Chapters~5--7]{heraldkirklivingston} to construct representations $ \varphi_\chi : \pi_1(X_K) \ra \GL(3, \Q(\zeta_7)[t^{\pm 1}]) $ of the knot group of $ K $ that determine twisted Alexander polynomials for each character in Table~\ref{tab:characters}. Fix a basepoint $ x_0 $ in $ S^3 \setminus \nu(F) $ and let $ \tilde{x}_0 $ be its lift to the preferred copy of $ S^3 \setminus \nu(F) $ in $ \widetilde{X_K^3} $, the triple cyclic cover of the knot complement $ X_K $. Also fix a based meridian $ \mu_0 $ in $ S^3 \setminus K $ and let $ \varepsilon : \pi_1(X_K) \ra \Z $ be the abelianisation homomorphism. Define a map $ l: \ker \varepsilon \ra H_1(\Sigma_3(K)) $ that takes a simple closed curve $ \gamma \subset S^3 \setminus K $ based at $ x_0 $ with $ \mathrm{lk}(K, \gamma) = 0 $ to the homology class of the well-defined lift $ \tilde{\gamma} $ in $ \widetilde{X_K^3} \subset \Sigma_3(K) $ based at $ \tilde{x}_0 $. In particular, $ l $ has the property that for any $ \gamma \in \ker \varepsilon $, we have
	\[\tag{$\ddagger$}
		l(\mu_0 \gamma \mu_0^{-1}) = t \cdot l(\gamma).
	\]
	Now consider the semidirect product $ \Z \ltimes H_1(\Sigma_3(K)) $, with $ \Z = \langle t \rangle $, whose product structure is given by
	$
	(t^{m_1}, x_1) \cdot (t^{m_2}, x_2) = (t^{m_1 + m_2}, t^{-m_2} \cdot x_1 + x_2)
	$ with $ t $ acting on elements of $ H_1(\Sigma_3(K)) $ by deck transformations. Fix a Wirtinger presentation of $ \pi_1(X_K) \cong \langle g_1, \dots, g_n \mid r_1, \dots, r_n \rangle$ and define a homomorphism
	\begin{align*}
		\psi : \pi_1(X_K) &\ra \Z \ltimes H_1(\Sigma_3(K)) \\
		g_i &\mapsto (t, l(\mu_0^{-1} g_i)) =: (t, v_i)
	\end{align*}
	on the generators of $ \pi_1(X_K) $, since clearly $ \mu_0^{-1} g_i \in \ker \varepsilon $. Observe that a relation $ g_i g_j g_i^{-1} g_k^{-1} = 1 $ imposes, via the group structure on $ \Z \ltimes H_1(\Sigma_3(K)) $, the condition
	\[
	\tag{$\ddagger\ddagger$}
		(1 - t) v_i + tv_j - v_k = 0.
	\]
	Finally, for a character $ \chi : H_1(\Sigma_3(K)) \ra \Z/7 $, we obtain a representation $ \varphi_\chi : \pi_1(X_K) \ra \GL(3, \Q(\zeta_7)[t^{\pm 1}]) $ by setting $ \varphi_\chi = \tau_\chi \circ \psi $, where
	\begin{align*}
		\tau_\chi : \Z \ltimes H_1(\Sigma_3(K)) &\ra \GL(3, \Q(\zeta_7)[t^{\pm 1}]) \\
		(t^m, v) &\mapsto \begin{pmatrix}0&1&0\\0&0&1\\t&0&0\end{pmatrix}^m \begin{pmatrix}\zeta_7^{\chi(v)}&0&0\\0&\zeta_7^{\chi(t\cdot v)}&\\0&0&\zeta_7^{\chi(t^2\cdot v)}\end{pmatrix}.
	\end{align*}

	We shall apply the equation $ (\ddagger) $ to determine the form of the first few $ v_k $ for $ K $ in terms of the generators $ \{a, b\} $ of $ H_1(\Sigma_3(K)) $ and then deduce the rest of $ v_k $ using $ (\ddagger\ddagger) $, giving us the desired $ \varphi_\chi $. We illustrate the process in more detail for $ K_1 $, with $ K_2 $ and $ K_3 $ cases being analogous.

	\begin{figure}[h]
		\centering
		\def\svgwidth{0.9\textwidth}
		\makebox[\textwidth][c]{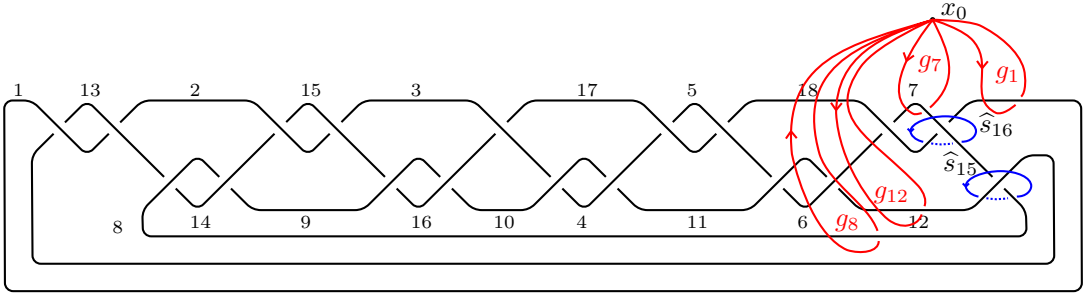}
		\caption{
			Choice of arc labels for $ K_1 $ giving the Wirtinger presentation of $ \pi_1(X_1) $.
		} \label{fig:K1_wirtinger}
	\end{figure}

	Recall that we orient $ K_1 $ clockwise. Index the arcs in the diagram of $ K_1 $ as shown in Figure~\ref{fig:K1_wirtinger}, starting with $ 1 $ at the top left and increasing the index at every undercrossing. This yields the following Wirtinger presentation of $ \pi_1(X_1) $, with generators being the meridians $ g_i $ about each arc $ i $ based at $ x_0 $:
	\begin{align*}
	\pi_1(X_1) = \left\langle
	g_1, \dots, g_{18}
	\left|
		\begin{array}{c c c c c}
			g_{1} g_{13} g_1^{-1} g_{12}^{-1}, & g_{3} g_{17} g_{3}^{-1} g_{16}^{-1}, & g_{7} g_{1} g_{7}^{-1} g_{18}^{-1}, & g_{16} g_{9} g_{16}^{-1} g_{10}^{-1}, & \\
			g_{13} g_{2} g_{13}^{-1} g_{1}^{-1}, & g_{17} g_{5} g_{17}^{-1} g_{4}^{-1}, & g_{8} g_{13} g_{8}^{-1} g_{14}^{-1}, & g_{10} g_{3} g_{10}^{-1} g_{4}^{-1}, & g_{6} g_{11} g_{6}^{-1} g_{12}^{-1}, \\
			g_{2} g_{15} g_{2}^{-1} g_{14}^{-1}, & g_{5} g_{18} g_{5}^{-1} g_{17}^{-1}, & g_{14} g_{8} g_{14}^{-1} g_{9}^{-1}, & g_{4} g_{10} g_{4}^{-1} g_{11}^{-1}, & g_{12} g_{7} g_{12}^{-1} g_{8}^{-1} \\
			g_{15} g_{3} g_{15}^{-1} g_{2}^{-1}, & g_{18} g_{7} g_{18}^{-1} g_{6}^{-1}, & g_{9} g_{15} g_{9}^{-1} g_{16}^{-1}, & g_{11} g_{5} g_{11}^{-1} g_{6}^{-1}, &
		\end{array}
	\right.
	\right\rangle
	\end{align*}
	Observe that $ \widehat{s}_{15} = g_8 g_{12}^{-1} $ and $ \widehat{s}_{16} = g_1^{-1} g_{7} $. Fix $ \mu_0 = g_1 $. Then $ v_1 = l(g_1^{-1} g_1) = 0 $ and $ v_7 = l(g_1^{-1} g_7) = b $. Also, using the property $ (\ddagger) $, we have
	\begin{align*}
		a = l(g_8 g_{12}^{-1}) &= l(g_8 g_1^{-1} g_1 g_{12}^{-1}) \\ &= l(g_8 g_1^{-1}) + l(g_1 g_{12}^{-1}) \\ &= l(g_1 g_1^{-1} g_8 g_1^{-1}) - l(g_{12} g_1^{-1}) \\ &= l(g_1 g_1^{-1} g_8 g_1^{-1}) - l(g_1 g_1^{-1} g_{12} g_1^{-1}) \\ &= tv_{8} - tv_{12}.
	\end{align*}
	Applying $ (\ddagger\ddagger) $ to the relation $ g_{12} g_{7} g_{12}^{-1} g_{8}^{-1} = 1 $ and recalling we are working modulo $ t^2 + t + 1 $, we get
	\[
		\begin{array}{c r l}
		&(1 - t) v_{12} + tv_7 - v_8 &= 0 \\
		\Longrightarrow &(1 - t) v_{12} - v_8 &= - tb \mid \cdot\, (-t) \\
		\Longrightarrow & (tv_{8} - tv_{12}) + t^2 v_{12} &= t^2 b \\
		\Longrightarrow & a + t^2 v_{12} &= t^2 b \mid \cdot\, t \\
		\Longrightarrow & v_{12} &= -ta + b.
	\end{array}
	\]
	Now we can use $ (\ddagger\ddagger) $ repeatedly to obtain the values of all $ v_i $. With the same conventions and the choice $ \mu_0 = g_1 $, for $ K_2 $ we have $ l(\widehat{s}_{14}) = l(g_1^{-1} g_6) = a $ and $ l(\widehat{s}_{16}) = l(g_{14} g_{7}^{-1}) = b$, while for $ K_3 $, $ l(\widehat{s}_{15}) = l(g_1^{-1} g_7) = a $ and $ l(\widehat{s}_{16}) = l(g_{8} g_{13}^{-1}) = b$; this lets us calculate the following values of $ v_i $ analogously:
	\begin{center}{}
	\begin{table}[h]
		\centering
		\begin{tabular}{c|ccc}
		& $\pi_1(X_1)$ & $\pi_1(X_2)$ & $\pi_1(X_3)$\\
		\cline{1-4}
		$v_1$ & $0$ & $0$ & $0$ \\
		$v_2$ & $(6t+5)a + (5t+6)b$ & $(5t+6)a + (4t+4)b$ & $(5t+6)a + (6t+5)b$ \\
		$v_3$ & $5ta + 5b$ & $3a + (3t+1)b$ & $(4t+3)a + (t+1)b$ \\
		$v_4$ & $(2t+5)a + 6b$ & $(2t+6)a + 2b$ & $(6t+3)a + b$ \\
		$v_5$ & $(6t+5)a + (5t+3)b$ & $(4t+1)a + (6t+5)b$ & $(6t+4)a + (4t+6)b$ \\
		$v_6$ & $5tb$ & $a$ & $(4t+1)a + (t+6)b$ \\
		$v_7$ & $b$ & $a + (6t+1)b$ & $a$ \\
		$v_8$ & $(5t+6)a + b$ & $(6t+6)a + (6t+5)b$ & $a + (5t+6)b$ \\
		$v_9$ & $(3t+2)a + (4t+1)b$ & $5ta + (3t+5)b$ & $(3t+6)a + (5t+3)b$ \\
		$v_{10}$ & $(t+2)a + (5t+1)b$ & $(2t+3)a + (3t+3)b$ & $(4t+6)a + (3t+3)b$ \\
		$v_{11}$ & $6a + (4t+1)b$ & $(3t+6)a + 5b$ & $(3t+6)a + 2tb$ \\
		$v_{12}$ & $6ta + b$ & $(6t+2)a + (6t+6)b$ & $(6t+2)a + 6b$ \\
		$v_{13}$ & $6a + (6t+6)b$ & $a+b$ & $a + 6tb$ \\
		$v_{14}$ & $(3t+4)a + (6t+2)b$ & $a + 5tb$& $(6t+6)a + 6b$ \\
		$v_{15}$ & $3a + (2t+4)b$ & $(5t+3)a + 6b$& $(6t+2)a + (3t+4)b$ \\
		$v_{16}$ & $5a + (2t+3)b$ & $(5t+5)a + (3t+5)b$ & $(t+1)a + (2t+6)b$ \\
		$v_{17}$ & $4a + (2t+2)b$ & $ta + (5t+3)b$ & $ta + (2t+5)b$ \\
		$v_{18}$ & $(6t+1)b$ &$(6t+1)a$ & $(6t+1)a$
		\end{tabular}
		\caption{Values of $ v_k = l(\mu_0^{-1} g_k) \in H_1(\Sigma_3(K_i)) $.}
		\label{tab:lifts}
	\end{table}
	\end{center}
	\vspace{-2em}
	Constructing representations $\varphi_\chi$ for the characters in Subsection~\ref{subsec:characters} is now mechanical.

	\subsection{Calculating twisted Alexander polynomials} Again, let $ K \in \{ K_1, K_2, K_3 \} $ and fix the Wirtinger presentation of $ \pi_1(X_K) $ as in Subsection~\ref{subsec:wirtinger}. Given a representation $ \varphi_\chi : \pi_1(X_K) \ra \GL(3, \Q(\zeta_7)[t^{\pm 1}])$, let $ \Phi : \Z[\pi_1(X_K)] \ra \mathrm{Mat}_3(\Q(\zeta_7)[t^{\pm 1}]) $ be its natural extension to the group ring $ \Z[\pi_1(X_K)] $ taking values in the set of $ 3 \times 3 $ matrices with $\Q(\zeta_7)[t^{\pm 1}]$ coefficients. Let 
	\[
	\Psi = \left( \frac{\pp r_i}{\pp g_j} \right)_{i,j=1,\dots,18}
	\]
	be the Fox matrix for the Wirtinger presentation of $ \pi_1(X_K) $; the row of $ \Psi $ corresponding to the relation $ g_i g_j g_i^{-1} g_k^{-1} $ has $ 1 - g_k $ in the $ i^{\textrm{th}} $ column, $ g_i $ in the $ j^{\textrm{th}} $ column, $ -1 $ in the $ k^{\textrm{th}} $ column and zeros elsewhere. Write $ r(\Psi) $ for the reduced Fox matrix obtained by dropping the first row and column from $ \Psi $ and let $ \Phi_\chi $ be the $ 51 \times 51 $ matrix obtained by applying $ \Phi $ to $ r(\Psi) $ entrywise. By~\cite[Section~9]{heraldkirklivingston}, the reduced twisted Alexander polynomial $ \widetilde{\Delta}_K^\chi(t) $ of $ (K, \chi) $ (for non-trivial $\chi$) is given by
		\[
			\widetilde{\Delta}_K^\chi(t) = \frac{\det{\Phi_\chi}}{(t - 1)\det(\varphi_\chi(g_1) - I)}.
		\]
	Thus we obtain the 11 reduced twisted Alexander polynomials listed in Appendix A associated with our characters of interest.
		
	\subsection{Obstructing sliceness of $ K_i $.} To show that $ K_1 $, $ K_2 $ and $ K_3 $ are not slice, we use the following generalisation of the Fox-Milnor condition, due to Kirk and Livingston~\cite{kirk1999twisted}.

	\begin{thm}[\protect{\cite[Proposition~6.1]{kirk1999twisted}}]
		Let $ K \subset S^3 $ be a slice knot and fix distinct primes $ p $ and $ q $. Then there exists a covering transformation invariant metaboliser $ N $ in $ H_1(\Sigma_p(K)) $ such that the following condition holds: for every character $ \chi : H_1(\Sigma_p(K)) \ra \Z/q $ that vanishes on $ N $, the associated reduced twisted Alexander polynomial $ \widetilde{\Delta}_K^\chi(t) \in \Q(\zeta_q)[t^{\pm 1}] $ is a \emph{norm}, i.e., $ \widetilde{\Delta}_K^\chi(t) $ can be written as
			\[
				\widetilde{\Delta}_K^\chi(t) = \lambda t^k f(t) \overline{f(t)}
			\]
		for some $ \lambda \in \Q(\zeta_q) $, $ k \in \Z $ and $ \overline{f(t)} $ obtained from $ f(t) \in \Q(\zeta_q)[t^{\pm 1}] $ by the involution $ t \mapsto t^{-1} $, $ \zeta_q \mapsto \zeta_q^{-1}$.
	\end{thm}

	Using the routine implemented in \texttt{SnapPy}~\cite{SnapPy} for determining whether an element of $ \Q(\zeta_q)[t^{\pm 1}] $ is a norm, which relies on the \texttt{SageMath} algorithm for factoring polynomials over cyclotomic fields, we conclude via a calculation in the accompanying notebooks that none of the 11 polynomials in Appendix~A are norms. This implies that $ K_1 $, $ K_2 $ and $ K_3 $ are not slice.

	\appendix
	\section{Reduced twisted Alexander polynomials for $ K_1 $, $ K_2 $ and $ K_3 $}
	The following table contains reduced twisted Alexander polynomials for knots $ K_1 $, $ K_2 $ and $ K_3 $ associated to characters vanishing on the metabolisers of respective knots; for brevity, we write $ \zeta = \zeta_7 $ and $ \theta = \zeta_7 + \zeta_7^2 + \zeta_7^4 $.
	\begin{longtable}{c | c}
		$(K_i, \chi)$ & $ \widetilde{\Delta}_{K_i}^\chi(t) $ \\
		& \\[\dimexpr-\normalbaselineskip+5pt]
		\hline\hline\\[\dimexpr-\normalbaselineskip+5pt]
		$(K_1, \chi_0^\alpha)$        &  \begin{minipage}{0.85\textwidth}\centering $-t^{15} + \left(-2 \theta - 1\right) t^{14} + \left(-8 \theta - 3\right) t^{13} + 15 t^{12} + \left(-3 \theta + 48\right) t^{11} + \left(-8 \theta + 33\right) t^{10} + \left(-48 \theta + 34\right) t^{9} + 199 t^{8} + \left(48 \theta + 82\right) t^{7} + \left(8 \theta + 41\right) t^{6} + \left(3 \theta + 51\right) t^{5} + 15 t^{4} + \left(8 \theta + 5\right) t^{3} + \left(2 \theta + 1\right) t^{2} - t$\end{minipage}\\
		& \\[\dimexpr-\normalbaselineskip+5pt]
		\hline\\[\dimexpr-\normalbaselineskip+5pt]
		$(K_1, \chi_0^\beta)$         &  \begin{minipage}{0.85\textwidth} \centering $-t^{15} + \left(-4 \theta + 5\right) t^{14} + \left(24 \theta - 15\right) t^{13} + \left(-93 \theta - 14\right) t^{12} + \left(98 \theta + 11\right) t^{11} + \left(-2 \theta + 71\right) t^{10} + \left(-11 \theta - 154\right) t^{9} + 360 t^{8} + \left(11 \theta - 143\right) t^{7} + \left(2 \theta + 73\right) t^{6} + \left(-98 \theta - 87\right) t^{5} + \left(93 \theta + 79\right) t^{4} + \left(-24 \theta - 39\right) t^{3} + \left(4 \theta + 9\right) t^{2} - t$ \end{minipage}\\
		& \\[\dimexpr-\normalbaselineskip+5pt]
		\hline\\[\dimexpr-\normalbaselineskip+5pt]
		$(K_1,\chi_6^{\alpha\beta})$ & \begin{minipage}{0.85\textwidth} \centering $-t^{15} + \left(2 \zeta^{5} - \zeta^{4} + 4 \zeta^{3} - \zeta^{2} - 2 \zeta + 5\right) t^{14} + \left(-3 \zeta^{5} + 7 \zeta^{4} - 24 \zeta^{3} - 3 \zeta^{2} + 2 \zeta - 20\right) t^{13} + \left(7 \zeta^{5} - 67 \zeta^{4} + 41 \zeta^{3} - 8 \zeta^{2} - 35 \zeta + 7\right) t^{12} + \left(-45 \zeta^{5} + 52 \zeta^{4} - 38 \zeta^{3} + 3 \zeta^{2} - \zeta + 19\right) t^{11} + \left(68 \zeta^{5} + 51 \zeta^{4} + 114 \zeta^{3} + 24 \zeta^{2} + 95 \zeta + 63\right) t^{10} + \left(116 \zeta^{5} + 121 \zeta^{4} + 80 \zeta^{3} + 56 \zeta^{2} + 124 \zeta + 65\right) t^{9} + \left(149 \zeta^{5} - 3 \zeta^{4} - 3 \zeta^{3} + 149 \zeta^{2} + 19\right) t^{8} + \left(-68 \zeta^{5} - 44 \zeta^{4} - 3 \zeta^{3} - 8 \zeta^{2} - 124 \zeta - 59\right) t^{7} + \left(-71 \zeta^{5} + 19 \zeta^{4} - 44 \zeta^{3} - 27 \zeta^{2} - 95 \zeta - 32\right) t^{6} + \left(4 \zeta^{5} - 37 \zeta^{4} + 53 \zeta^{3} - 44 \zeta^{2} + \zeta + 20\right) t^{5} + \left(27 \zeta^{5} + 76 \zeta^{4} - 32 \zeta^{3} + 42 \zeta^{2} + 35 \zeta + 42\right) t^{4} + \left(-5 \zeta^{5} - 26 \zeta^{4} + 5 \zeta^{3} - 5 \zeta^{2} - 2 \zeta - 22\right) t^{3} + \left(\zeta^{5} + 6 \zeta^{4} + \zeta^{3} + 4 \zeta^{2} + 2 \zeta + 7\right) t^{2} - t$\end{minipage} \\
		& \\[\dimexpr-\normalbaselineskip+5pt]
		\hline\\[\dimexpr-\normalbaselineskip+5pt]
		$(K_1,\chi_4^{\beta\alpha})$ &  \begin{minipage}{0.85\textwidth} \centering $-t^{15} + \left(2 \zeta^{5} + \zeta^{4} + 2 \zeta^{3} + \zeta^{2} - \zeta + 2\right) t^{14} + \left(-5 \zeta^{5} - 2 \zeta^{4} - 3 \zeta^{3} - 6 \zeta^{2} - 2 \zeta - 9\right) t^{13} + \left(10 \zeta^{5} + 4 \zeta^{4} + 9 \zeta^{2} + 20\right) t^{12} + \left(-35 \zeta^{5} - 36 \zeta^{4} - 30 \zeta^{3} - 35 \zeta^{2} - 4 \zeta - 10\right) t^{11} + \left(44 \zeta^{5} - 10 \zeta^{4} + 8 \zeta^{3} + 47 \zeta^{2} + 52 \zeta + 85\right) t^{10} + \left(-57 \zeta^{5} - 17 \zeta^{4} - 63 \zeta^{3} + 29 \zeta^{2} - 27 \zeta + 11\right) t^{9} + \left(7 \zeta^{5} + 38 \zeta^{4} + 38 \zeta^{3} + 7 \zeta^{2} - 59\right) t^{8} + \left(56 \zeta^{5} - 36 \zeta^{4} + 10 \zeta^{3} - 30 \zeta^{2} + 27 \zeta + 38\right) t^{7} + \left(-5 \zeta^{5} - 44 \zeta^{4} - 62 \zeta^{3} - 8 \zeta^{2} - 52 \zeta + 33\right) t^{6} + \left(-31 \zeta^{5} - 26 \zeta^{4} - 32 \zeta^{3} - 31 \zeta^{2} + 4 \zeta - 6\right) t^{5} + \left(9 \zeta^{5} + 4 \zeta^{3} + 10 \zeta^{2} + 20\right) t^{4} + \left(-4 \zeta^{5} - \zeta^{4} - 3 \zeta^{2} + 2 \zeta - 7\right) t^{3} + \left(2 \zeta^{5} + 3 \zeta^{4} + 2 \zeta^{3} + 3 \zeta^{2} + \zeta + 3\right) t^{2} - t$ \end{minipage}\\
		& \\[\dimexpr-\normalbaselineskip+5pt]
		\hline\hline\\[\dimexpr-\normalbaselineskip+5pt]
		$(K_2, \chi_0^\alpha)$        &  \begin{minipage}{0.85\textwidth} \centering $t^{15} + \left(-\theta - 2\right) t^{14} + \left(-2 \theta - 1\right) t^{13} + \left(3 \theta + 3\right) t^{12} + \left(-13 \theta - 22\right) t^{11} + \left(-15 \theta - 5\right) t^{10} + \left(25 \theta + 13\right) t^{9} - 82 t^{8} + \left(-25 \theta - 12\right) t^{7} + \left(15 \theta + 10\right) t^{6} + \left(13 \theta - 9\right) t^{5} -3 \theta t^{4} + \left(2 \theta + 1\right) t^{3} + \left(\theta - 1\right) t^{2} + t$\end{minipage}\\
		& \\[\dimexpr-\normalbaselineskip+5pt]
		\hline\\[\dimexpr-\normalbaselineskip+5pt]
		$(K_2,\chi_0^\beta)$         &  \begin{minipage}{0.85\textwidth} \centering $t^{15} + \left(-4 \theta - 7\right) t^{14} + \left(16 \theta + 15\right) t^{13} + \left(-41 \theta - 26\right) t^{12} + \left(55 \theta + 5\right) t^{11} + \left(-20 \theta - 18\right) t^{10} + \left(-25 \theta + 114\right) t^{9} - 292 t^{8} + \left(25 \theta + 139\right) t^{7} + \left(20 \theta + 2\right) t^{6} + \left(-55 \theta - 50\right) t^{5} + \left(41 \theta + 15\right) t^{4} + \left(-16 \theta - 1\right) t^{3} + \left(4 \theta - 3\right) t^{2} + t$\end{minipage}\\
		& \\[\dimexpr-\normalbaselineskip+5pt]
		\hline\\[\dimexpr-\normalbaselineskip+5pt]
		$(K_2,\chi_1^{\alpha\beta})$ &  \begin{minipage}{0.85\textwidth} \centering $t^{15} + \left(-3 \zeta^{5} + 3 \zeta^{4} - 2 \zeta^{3} + \zeta^{2} - 4\right) t^{14} + \left(4 \zeta^{5} - 12 \zeta^{4} + 6 \zeta^{3} - 13 \zeta^{2} + \zeta\right) t^{13} + \left(23 \zeta^{4} + 9 \zeta^{3} + 30 \zeta^{2} - 4 \zeta + 17\right) t^{12} + \left(-49 \zeta^{5} - 17 \zeta^{4} - 50 \zeta^{3} - 46 \zeta^{2} - 33 \zeta - 13\right) t^{11} + \left(-48 \zeta^{5} + 5 \zeta^{4} + 67 \zeta^{3} - 34 \zeta^{2} + 87 \zeta - 36\right) t^{10} + \left(164 \zeta^{5} + 69 \zeta^{4} + 127 \zeta^{3} + 39 \zeta^{2} + 83 \zeta + 75\right) t^{9} + \left(173 \zeta^{5} + 32 \zeta^{4} + 32 \zeta^{3} + 173 \zeta^{2} + 166\right) t^{8} + \left(-44 \zeta^{5} + 44 \zeta^{4} - 14 \zeta^{3} + 81 \zeta^{2} - 83 \zeta - 8\right) t^{7} + \left(-121 \zeta^{5} - 20 \zeta^{4} - 82 \zeta^{3} - 135 \zeta^{2} - 87 \zeta - 123\right) t^{6} + \left(-13 \zeta^{5} - 17 \zeta^{4} + 16 \zeta^{3} - 16 \zeta^{2} + 33 \zeta + 20\right) t^{5} + \left(34 \zeta^{5} + 13 \zeta^{4} + 27 \zeta^{3} + 4 \zeta^{2} + 4 \zeta + 21\right) t^{4} + \left(-14 \zeta^{5} + 5 \zeta^{4} - 13 \zeta^{3} + 3 \zeta^{2} - \zeta - 1\right) t^{3} + \left(\zeta^{5} - 2 \zeta^{4} + 3 \zeta^{3} - 3 \zeta^{2} - 4\right) t^{2} + t$\end{minipage}\\
		& \\[\dimexpr-\normalbaselineskip+5pt]
		\hline\\[\dimexpr-\normalbaselineskip+5pt]
		$(K_2,\chi_1^{\beta\alpha})$ &  \begin{minipage}{0.85\textwidth} \centering $t^{15} + \left(-\zeta^{5} - 2 \zeta^{4} + \zeta^{2} - 3 \zeta - 7\right) t^{14} + \left(4 \zeta^{5} + 8 \zeta^{4} - 4 \zeta^{3} - 4 \zeta^{2} + 17 \zeta + 28\right) t^{13} + \left(-\zeta^{5} - 20 \zeta^{4} + 21 \zeta^{3} + 30 \zeta^{2} - 52 \zeta - 78\right) t^{12} + \left(-10 \zeta^{5} + 38 \zeta^{4} - 51 \zeta^{3} - 88 \zeta^{2} + 122 \zeta + 187\right) t^{11} + \left(81 \zeta^{5} - 15 \zeta^{4} + 87 \zeta^{3} + 205 \zeta^{2} - 155 \zeta - 358\right) t^{10} + \left(-256 \zeta^{5} - 31 \zeta^{4} - 157 \zeta^{3} - 312 \zeta^{2} + 91 \zeta + 487\right) t^{9} + \left(434 \zeta^{5} + 146 \zeta^{4} + 146 \zeta^{3} + 434 \zeta^{2} - 430\right) t^{8} + \left(-403 \zeta^{5} - 248 \zeta^{4} - 122 \zeta^{3} - 347 \zeta^{2} - 91 \zeta + 396\right) t^{7} + \left(360 \zeta^{5} + 242 \zeta^{4} + 140 \zeta^{3} + 236 \zeta^{2} + 155 \zeta - 203\right) t^{6} + \left(-210 \zeta^{5} - 173 \zeta^{4} - 84 \zeta^{3} - 132 \zeta^{2} - 122 \zeta + 65\right) t^{5} + \left(82 \zeta^{5} + 73 \zeta^{4} + 32 \zeta^{3} + 51 \zeta^{2} + 52 \zeta - 26\right) t^{4} + \left(-21 \zeta^{5} - 21 \zeta^{4} - 9 \zeta^{3} - 13 \zeta^{2} - 17 \zeta + 11\right) t^{3} + \left(4 \zeta^{5} + 3 \zeta^{4} + \zeta^{3} + 2 \zeta^{2} + 3 \zeta - 4\right) t^{2} + t$\end{minipage}\\
		& \\[\dimexpr-\normalbaselineskip+5pt]
		\hline\hline\\[\dimexpr-\normalbaselineskip+5pt]
		$ \begin{array}{c}
		(K_3,\chi_0^\alpha) =  \\ (K_3,\chi_2^{\alpha\beta})
		\end{array} $       
		&  \begin{minipage}{0.85\textwidth} \centering $t^{15} + \left(\theta - 3\right) t^{14} + \left(-3 \theta - 1\right) t^{13} + \left(-2 \theta - 22\right) t^{12} + \left(-73 \theta - 8\right) t^{11} + \left(10 \theta + 239\right) t^{10} + \left(362 \theta + 223\right) t^{9} - 675 t^{8} + \left(-362 \theta - 139\right) t^{7} + \left(-10 \theta + 229\right) t^{6} + \left(73 \theta + 65\right) t^{5} + \left(2 \theta - 20\right) t^{4} + \left(3 \theta + 2\right) t^{3} + \left(-\theta - 4\right) t^{2} + t$\end{minipage}\\
		& \\[\dimexpr-\normalbaselineskip+5pt]
		\hline\\[\dimexpr-\normalbaselineskip+5pt]
		$(K_3,\chi_0^\beta)$         &  \begin{minipage}{0.85\textwidth} \centering $t^{15} - 7 t^{14} + \left(-2 \theta + 17\right) t^{13} + \left(6 \theta - 32\right) t^{12} + \left(-26 \theta + 26\right) t^{11} + \left(24 \theta + 8\right) t^{10} + \left(40 \theta + 83\right) t^{9} - 178 t^{8} + \left(-40 \theta + 43\right) t^{7} + \left(-24 \theta - 16\right) t^{6} + \left(26 \theta + 52\right) t^{5} + \left(-6 \theta - 38\right) t^{4} + \left(2 \theta + 19\right) t^{3} - 7 t^{2} + t$\end{minipage}\\
		& \\[\dimexpr-\normalbaselineskip+5pt]
		\hline\\[\dimexpr-\normalbaselineskip+5pt]
		$(K_3,\chi_3^{\beta\alpha})$ & \begin{minipage}{0.85\textwidth} \centering $t^{15} + \left(-\zeta^{5} + 3 \zeta^{4} + 2 \zeta^{3} + 2 \zeta^{2} + 4 \zeta - 3\right) t^{14} + \left(18 \zeta^{5} + \zeta^{4} + 3 \zeta^{3} + 3 \zeta^{2} - 4 \zeta + 11\right) t^{13} + \left(-33 \zeta^{5} - 17 \zeta^{4} - 26 \zeta^{3} - 21 \zeta^{2} - 11 \zeta - 60\right) t^{12} + \left(-5 \zeta^{5} - 52 \zeta^{4} - 16 \zeta^{3} - 3 \zeta^{2} - 56 \zeta + 45\right) t^{11} + \left(-14 \zeta^{5} + 48 \zeta^{4} + 66 \zeta^{3} - 18 \zeta^{2} + 59 \zeta - 5\right) t^{10} + \left(106 \zeta^{5} + 89 \zeta^{4} - 10 \zeta^{3} + 109 \zeta^{2} + 18 \zeta + 101\right) t^{9} + \left(-133 \zeta^{5} - 123 \zeta^{4} - 123 \zeta^{3} - 133 \zeta^{2} - 212\right) t^{8} + \left(91 \zeta^{5} - 28 \zeta^{4} + 71 \zeta^{3} + 88 \zeta^{2} - 18 \zeta + 83\right) t^{7} + \left(-77 \zeta^{5} + 7 \zeta^{4} - 11 \zeta^{3} - 73 \zeta^{2} - 59 \zeta - 64\right) t^{6} + \left(53 \zeta^{5} + 40 \zeta^{4} + 4 \zeta^{3} + 51 \zeta^{2} + 56 \zeta + 101\right) t^{5} + \left(-10 \zeta^{5} - 15 \zeta^{4} - 6 \zeta^{3} - 22 \zeta^{2} + 11 \zeta - 49\right) t^{4} + \left(7 \zeta^{5} + 7 \zeta^{4} + 5 \zeta^{3} + 22 \zeta^{2} + 4 \zeta + 15\right) t^{3} + \left(-2 \zeta^{5} - 2 \zeta^{4} - \zeta^{3} - 5 \zeta^{2} - 4 \zeta - 7\right) t^{2} + t$\end{minipage}
		\end{longtable}

	\printbibliography


\end{document}